\documentclass[12pt,dvipsnames]{article}
\usepackage{a4} 

\usepackage{amsthm,amsmath,amssymb,amsfonts}
\usepackage{mathtools}

\usepackage{xcolor}
\usepackage{graphics}
\usepackage{graphicx}
\usepackage{subcaption}
\usepackage{float}
\usepackage{booktabs, multirow}
\usepackage{url}
\usepackage{verbatim,upquote}
\usepackage{algpseudocode}
\usepackage{hyperref}
\usepackage[section]{algorithm}
\usepackage[nameinlink]{cleveref}

\usepackage[inner=3cm,outer=3cm,top=2.5cm,bottom=2.5cm]{geometry}
\hypersetup{colorlinks,urlcolor=blue,citecolor=hotpink,linkcolor=blue}

\graphicspath{{./}}
\DeclareGraphicsExtensions{.pdf,.eps,.png,.jpg}

\newtheorem{theorem}{Theorem}
\newtheorem*{theorem*}{Theorem}
\newtheorem{lemma}[theorem]{Lemma}
\newtheorem{corollary}[theorem]{Corollary}
\newtheorem{proposition}[theorem]{Proposition}

\numberwithin{equation}{section}
\numberwithin{table}{section}
\numberwithin{figure}{section}
\numberwithin{theorem}{section}

\definecolor{ForestGreen}{cmyk}{0.91,0,0.88,0.12}
\definecolor{hotpink}{rgb}{0.9,0,0.5}
\hypersetup{colorlinks,urlcolor=blue,citecolor=hotpink,linkcolor=ForestGreen}

\def\e{\ensuremath{\mathrm{e}}}

\def\>{\mskip\medmuskip}

\DeclarePairedDelimiter\norm{\lVert}{\rVert}

\newcommand{\normF}[1]{\ensuremath{\norm{#1}_F}}

\def\normt#1{\|#1\|_{2}}

\def\C{\mathbb{C}}

\DeclareMathOperator{\diag}{diag}

\DeclareMathOperator{\spn}{span}
\DeclareMathOperator{\range}{range}

\DeclareMathOperator{\conv}{conv}

\newcommand{\bb}{\ensuremath{\overline{b}}}
\newcommand{\KX}{\ensuremath{K_{\rm X}}}
\newcommand{\BX}{\ensuremath{B_{\rm X}}}
\newcommand{\cX}{\ensuremath{c_{\rm X}}}
\newcommand{\QX}{\ensuremath{Q_{\rm X}}}
\newcommand{\MX}{\ensuremath{M_{\rm X}}}
\newcommand{\RX}{\ensuremath{R_{\rm X}}}
\newcommand{\HX}{\ensuremath{H_{\rm X}}}
\newcommand{\X}{\ensuremath{\rm X}}

\newcommand{\XK}{\ensuremath{X_{\rm K}}}
\newcommand{\QK}{\ensuremath{Q_{\rm K}}}

\newcommand{\MK}{\ensuremath{M_{\rm K}}}
\newcommand{\GK}{\ensuremath{G_{\rm K}}}

\newcommand{\PiW}{\ensuremath{\Pi_{\rm W}}}
\newcommand{\PiK}{\ensuremath{\Pi_{\rm K}}}

\newcommand{\F}{\ensuremath{\mathcal F}}

\renewenvironment{abstract}{{\small \noindent\textbf{Abstract.}}}{\smallskip}
\newenvironment{keywords}{{\small \noindent\textbf{Keywords.}}}{\smallskip}
\newenvironment{MSCcodes}{{\small \noindent\textbf{MSC codes.}}}{\smallskip}
\date{\vskip -30pt}

\title{Field-of-values analysis of augmented Krylov methods for matrix \(\varphi\)-function actions\thanks{%
		Version of July 20, 2026. 
	}
}
\author{Xiaobo Liu%
	\thanks{%
		Max Planck Institute for Dynamics of Complex Technical Systems,
		Magdeburg, 39106, Germany
		(\texttt{xliu@mpi-magdeburg.mpg.de}).
	}
	\and 
	Marcel Schweitzer%
	\thanks{%
		School of Mathematics and Natural Sciences, Bergische Universit\"{a}t Wuppertal, 42097 Wuppertal, Germany
		(\texttt{marcel@uni-wuppertal.de}).
	}
}

\makeatother

\begin{document}
\maketitle

\begin{abstract}
We revisit established Krylov subspace methods for linear combinations of matrix $\varphi$-function actions from the viewpoint of the block triangular formulation of Al-Mohy and Liu [SIAM J. Sci. Comput., 48 (2026), pp. A726--A747]. 
In algorithms such as KIOPS [J. Comput. Phys., 372 (2018), pp. 236--255], one uses an augmentation approach based on evaluating the exponential of a slightly larger matrix that contains the operant vectors in its off-diagonal block, and its field of values may therefore grow substantially with these vectors. Typical convergence estimates for Krylov subspace methods result from bounding the error of polynomial approximations for the exponential on the field of values, so that only very pessimistic convergence estimates are available for these methods, in spite of their good practical performance. 
In contrast, the larger block formulation established by Al-Mohy and Liu involves an operator whose field of values is independent of the operant vectors, leading to more favorable convergence bounds. We work out the details of how these two approaches are connected to each other, which allows us to transfer the convergence bounds from the latter to the former, thus better explaining the observed performance.
\end{abstract}



\section{Introduction}\label{sect:intro}
Exponential integrators are widely used for stiff and highly oscillatory systems of ordinary differential equations
\begin{equation}\label{ode}
	\frac{\mathrm{d}y}{\mathrm{d}t} = Ay(t) + g(t,y(t)),\quad
	y(t_0)=y_0,\quad y\in\C^n,\quad t\ge t_0,
\end{equation}
where $A\in\C^{n\times n}$ is typically a large sparse Jacobian or a linearized spatial discretization of a PDE~\cite{coma02,hls98,hoos10}.  Their construction is based on the variation-of-constants formula
\begin{equation}\label{eq:variation_of_constants}
	y(t)=\e^{(t-t_0)A}y_0+
	\int_{t_0}^t \e^{(t-\tau)A}g(\tau,y(\tau))\,\mathrm d\tau.
\end{equation}
When the nonlinear term in~\eqref{eq:variation_of_constants} is expanded locally~\cite[Lem.~5.1]{miwr05}, the resulting integrators involve the exponential-like matrix functions~\cite{high:FM, hoos10}
\begin{equation}\label{def.phik}
	\varphi_k(A)=\sum_{j=0}^{\infty}\frac{A^j}{(j+k)!},
	\quad k\ge0,
\end{equation}
with $\varphi_0(A)=\e^A$. Hence a central computational kernel is the linear combination of actions
\begin{equation}\label{eq:phi-lin-comb}
	y:=\varphi_0(A)b_0+\varphi_1(A)b_1+\cdots+\varphi_s(A)b_s,
\end{equation}
where the vectors $b_j$ depend on the specific time-stepping formula that is employed. We assume throughout that $b_s\ne0$; otherwise, we can redefine $s$ as the largest index $j$ for which $b_j\ne0$.

Evaluating the kernel~\eqref{eq:phi-lin-comb} typically incurs the dominant cost in exponential integrator implementations, and has motivated various methods for computing the $\varphi$-function actions; see, e.g.,~\cite{almo25, almo25a, alhi11, bkt21, ccz23, det23, grt18, krsc25, niwr12} and the references therein.

An influential approach to evaluate~\eqref{eq:phi-lin-comb} is to embed it into the exponential of an augmented matrix. This approach was first considered for the special case of computing the action of $\varphi_1(A)$ in~\cite{saad92}, and was later generalized to further cases in~\cite{niwr12,sidj98} and finally in~\cite[Thm.~1]{grt18}. Specifically, the KIOPS algorithm~\cite{grt18} uses the $(n+s)$-dimensional matrix

\begin{equation}\label{eq:aug-mat-kiops}
K = \begin{bmatrix}
	A & B \\
	0 & J_s(0)
\end{bmatrix}, \quad 
B = \begin{bmatrix}
	b_s & b_{s-1} & \ldots & b_1
\end{bmatrix}\in \C^{n\times s},
\end{equation}
where $J_s(0)$ is the $s$-dimensional Jordan block associated with eigenvalue zero, and then applies an adaptive Krylov method with incomplete orthogonalization. KIOPS is widely regarded as a fast, state-of-the-art matrix-free Krylov solver for the linear combination~\eqref{eq:phi-lin-comb}; see~\cite{almo25, ccz23, det23}.
At the same time, the augmented matrix underlying KIOPS contains the operant vectors $b_j$, $j=1,2,\ldots,s$, in its off-diagonal block. Therefore, the Euclidean norm and field of values of $K$ are properties not only of $A$ and $s$, but also of the particular operant vectors. This can make standard field-of-values convergence estimates like~\cite[Prop.~3.1]{bere09} look much worse than the observed performance of KIOPS.

A different block formulation was derived by Al-Mohy and Liu~\cite[Thm.~2.1]{alli26} for the simultaneous computation of the matrix $\varphi$-functions themselves, rather than their actions on vectors. The augmented matrix in this approach is of size $(s+1)n \times (s+1)n$ and has the form
\begin{equation}\label{eq:aug-mat-blk}
	W=\begin{bmatrix}
		A & E \\ 0 & J
	\end{bmatrix}, \quad
	E = \begin{bmatrix}
		I_n & 0 & \ldots & 0
	\end{bmatrix}\in \C^{n\times sn}, \quad 
	J = J_s(0) \otimes I_n,
\end{equation}
where $I_n$ denotes the identity matrix of order $n$ and $\otimes$ denotes the Kronecker product. We stress that this larger augmentation is not attractive as the basis for a direct matrix-free Krylov implementation, because its auxiliary block has dimension $sn$, thus substantially increasing the length of vectors one would have to work with. Its advantage is analytical: in contrast to $K$ in~\eqref{eq:aug-mat-kiops}, the operator $W$ is independent of the vectors $b_1,\ldots,b_s$. Thus its field of values can be bounded in terms of the field of values of $A$, the nilpotent block $J$, and the fixed coupling $E$ with $\normt{E}=1$.

The purpose of this paper is to connect these two viewpoints.  We show that the KIOPS matrix $K$ is not a different augmented operator in the polynomial approximation problem. Rather, it is the representation of the block matrix $W$ restricted to the invariant subspace
\begin{equation}\label{eq:bb}
	\C^n\oplus \mathcal K_s(J,\bb),\quad
	\bb=[b_1^*,\ldots,b_s^*]^* \in\C^{ns}, 
\end{equation}
written in the nonorthogonal Jordan-chain basis $[J^{s-1}\bb,\ldots,J\bb,\bb]$ for the lower Krylov space.  Hence every polynomial in $K$ acting on the KIOPS starting vector gives, after projection onto the first $n$ entries, exactly the same vector as the corresponding polynomial in $W$ acting on the starting vector $\bb$. 
The coordinate map from KIOPS variables into the larger block formulation induces an inner product and motivates a different metric on field-of-values convergence bounds. This metric field of values of $K$ is contained in the standard field of values of $W$ and avoids the pessimistic right-hand side dependent enlargement seen in the Euclidean field of values of the KIOPS matrix. 

The rest of this paper is organized as follows. We start in~\Cref{sect:poly-krylov} with a brief review of polynomial Krylov approximations and the two augmented formulations for evaluating~\eqref{eq:phi-lin-comb}.~\Cref{sect:basis-framework} develops the basis-dependent block formulation and recovers the KIOPS matrix as a particular coordinate choice. Particularly, the conditioning of the KIOPS basis is discussed.~\Cref{sect:basis-metric} introduces the metric induced by the basis and characterizes the associated metric distortion, including the effect of common right-hand-side scaling. In~\Cref{sect:fov}, we transfer field-of-values polynomial approximation bounds from the larger block matrix to the reduced coordinate formulation via the induced metric and specialize the result to KIOPS. The resulting convergence bounds are illustrated in numerical experiments in~\Cref{sect:experiments}. Conclusions are drawn in~\Cref{sect:conclusions}.

\section{Polynomial Krylov subspace method}
\label{sect:poly-krylov}

For a given function $f$ and vector $b$, the polynomial Krylov method approximates $f(A)b$ in the Krylov subspace $\mathcal{K}_m(A, b) =  \spn\{ b, Ab, \dots, A^{m-1}b \}$, for which the backbone is the Arnoldi method~\cite{arno51}, presented in~\Cref{alg:arnoldi}.\footnote{We note that when $A$ is Hermitian, the Arnoldi method turns into the short-recurrence Lanczos method~\cite{lanc50}, which we do not cover here in detail as the considered augmentation techniques necessarily lead to non-Hermitian operators.} Initialized with $v_1 = b / \normt{b}$, it builds an orthonormal basis $V_m= [v_1,v_2,\dots,v_m]$ of $\mathcal{K}_m(A, b)$ which fulfills the Arnoldi relation
\begin{equation}\label{eq:arnoldi}
AV_m = V_m H_m + h_{m+1,m} v_{m+1} e_m^*,
\end{equation}
where $H_m$ is $m\times m$ upper Hessenberg and $e_m$ is the last column of the $m\times m$ identity matrix.
The resulting Arnoldi approximation $y_m \approx f(A)b$ is obtained by Galerkin projection onto $\mathcal{K}_m(A, b)$~\cite[Eq.~(2.3)]{holu97} and evaluation of $f$ on the $m\times m$ matrix $H_m$,
\begin{equation}\label{eq:fab-arnoldi}
	y_m:= V_m f(V_m^{\dagger} A V_m)V_m^\dagger b = V_m f(H_m)e_1\normt{b},
\end{equation}
where $V_m^\dagger = (V_m^*V_m)^{-1}V_m^*$ denotes the Moore--Penrose pseudoinverse~\cite{penr55} of $V_m$ (which coincides with $V_m^*$ when $V_m$ has orthonormal columns).

A standard way of obtaining bounds for the error of the Arnoldi approximation is by relating it to the error of the best uniform polynomial approximations of $f$ on the field of values 
$$
	\F(A)=\left\{ \frac{z^* Az}{z^*z}\colon z\neq 0 \right\} = \left\{ z^* Az \colon \|z\|_2 = 1 \right\};
$$
see, e.g.,~\cite{bere09,benz20,holu97,waye17}. In particular, when $f$ is analytic on a neighborhood of $\F(A)$,
\begin{equation}\label{eq:krylov_convergence_bound}
    \|f(A)b-y_m\|_2 \leq 2C_{\rm cr} \|b\|_2 E_{m-1}(f,\F(A)),
\end{equation}
where for a compact set $\Omega$,
\begin{equation}\label{eq:error-best-poly}
E_{m-1}(f,\Omega) := \min_{p\in\mathcal P_{m-1}} \max_{\zeta\in\Omega}|f(\zeta)-p(\zeta)|,
\end{equation}
with $\mathcal P_{m-1}$ the space of polynomials of degree at most $m-1$, and where $C_{\rm cr} \le 1+\sqrt2$ is a constant independent of $A$. This result can be obtained straightforwardly from the famous Crouzeix--Palencia theorem, which states that
\begin{equation}\label{eq:crouzeix-ineq}
	\normt{f(A)} \le C_{\rm cr} \max_{\zeta\in\F(A)} |f(\zeta)|;
\end{equation}
see~\cite{crou07,crpa17}. 

Clearly, the bound~\eqref{eq:krylov_convergence_bound} deteriorates more and more for growing $\F(A)$.

\begin{algorithm}[t]
	\caption{Arnoldi iteration with modified Gram--Schmidt orthogonalization.}
	\label{alg:arnoldi}
	\begin{algorithmic}[1]
		\Statex \textbf{Input:} $A \in \mathbb{C}^{n\times n}$, $b \in \mathbb{C}^n$, number of Arnoldi steps $m$.
		\Statex \textbf{Output:} Orthonormal basis $[v_1,v_2,\dots,v_m]$ of $\mathcal{K}_m(A,b)$, upper Hessenberg matrix $H_m=(h_{i,j}) \in \mathbb{C}^{m\times m}$.
		\State $v_1 \gets b/\normt{b}$
		\For{$j = 1, \dots, m$}
		\State $w \gets A v_j$
		\For{$i = 1, \dots, j$}
		\State $h_{ij} \gets v_i^* w$
		\State $w \gets w - h_{ij} v_i$
		\EndFor
		\State $h_{j+1,j} \gets \|w\|_2$
		\If{$h_{j+1,j} = 0$}
		\State terminate \Comment{happy breakdown}
		\EndIf
		\State $v_{j+1} \gets w / h_{j+1,j}$
		\EndFor
	\end{algorithmic}
\end{algorithm}

\subsection{Two block formulations}

As mentioned in \Cref{sect:intro}, one can evaluate a linear combination of $\varphi$-functions~\eqref{eq:phi-lin-comb} by constructing a Krylov subspace associated with an augmented matrix. The KIOPS algorithm~\cite{grt18} evaluates~\eqref{eq:phi-lin-comb} by applying a polynomial Krylov approximation to the augmented exponential action, such that
\begin{equation}\label{eq:poly-krylov-KIOPS}
y= 
\begin{bmatrix}
	I_n & 0
\end{bmatrix} \e^{K} c, \quad 
c = \begin{bmatrix}
	b_0\\ e_s
\end{bmatrix},
\end{equation}
where $e_s$ is the last column of the $s$-by-$s$ identity matrix and $K$ is defined in~\eqref{eq:aug-mat-kiops}. In practice, the augmented matrix need not be assembled.  For $x\in\C^n$ and $w\in\C^s$,
$$
	K \begin{bmatrix}
		x \\ w
		\end{bmatrix} 
	= \begin{bmatrix}
		Ax+Bw \\ J_s(0) w
	\end{bmatrix}
	=
	\begin{bmatrix}
		Ax+\sum_{j=1}^sw_{s+1-j}b_j\\
		J_s(0)w
	\end{bmatrix}.
$$
Hence each Krylov step requires one product with $A$ and a short linear combination of the data vectors $b_j$. The algorithm further combines this matrix-free augmented action with adaptive substepping and incomplete orthogonalization, which is one reason for its observed efficiency in exponential integrator computations~\cite{almo25, ccz23, det23, grt18}.

On the other hand, with the block formulation~\eqref{eq:aug-mat-blk}, the vector~\eqref{eq:phi-lin-comb} is obtained from the first $n$ components of $\e^Wb$, where $b = [b_0^*, \bb^*]^*$, and $W$ and $\bb$ are defined in~\eqref{eq:aug-mat-blk} and~\eqref{eq:bb}, respectively; see~\cite[Thm.~2.1]{alli26}.

We may assume $\bb\ne0$, since if $\bb=0$, then no matrix $\varphi$-functions beyond the exponential are required in~\eqref{eq:phi-lin-comb}; this is a degenerate case that is not considered in this paper.

\section{Basis-dependent Krylov approximation}
\label{sect:basis-framework}

The lower-right block $J$ in $W$ from~\eqref{eq:aug-mat-blk} has size $sn\times sn$, but inherits the nilpotency of the $s$-dimensional Jordan block, so that the dimension of any Krylov subspace with respect to $J$ is bounded by $s$. More generally, we have the following lemma.

\begin{lemma}\label[lemma]{lem:lb-kry-dim}
Let $\bb$ be defined as in~\eqref{eq:bb} and let $m \geq s$. Then, if $b_s \neq 0$, we have $\dim\mathcal K_m(J, \bb)=s$.  More generally, if $q=\max\{j:b_j\ne0\}$, then $\dim\mathcal K_m(J, \bb)=q$.
\end{lemma}

\begin{proof}
For $k\ge s$, it holds $J^k = (J_s(0)\otimes I_n)^k = J_s(0)^k \otimes I_n = 0$, which directly yields $\dim\mathcal K_m(J,\bb)\le s$.

For $k=0,\ldots,s-1$, the Kronecker structure of $J=J_s(0)\otimes I_n$ gives
$$
	J^k\bb=
	\begin{bmatrix}
	 b_{k+1}^* & b_{k+2}^* & \ldots & b_s^* & 0^* & \ldots & 0^*
	\end{bmatrix}^*.
$$
Suppose first that $b_s\ne0$ and $\sum_{k=0}^{s-1}\alpha_kJ^k\bb=0$.  Inspecting the last block row of this vector identity gives $\alpha_0b_s=0$, hence $\alpha_0=0$.  
Inspecting the next block row gives $\alpha_1b_s=0$, and so $\alpha_1=0$.  
Continuing upward yields $\alpha_0=\cdots=\alpha_{s-1}=0$.  Hence the set of $s$ vectors $\{\bb,\ J\bb,\ \ldots,\ J^{s-1}\bb\}$ is linearly independent, which implies $\dim\mathcal K_s(J,\bb)=s$. If the last nonzero block of $\bb$ is $b_q$, the same argument applied to the truncated chain $\bb,\ J\bb,\ \ldots,\ J^{q-1}\bb$ gives $\dim\mathcal K_s(J, \bb)=q$.
\end{proof}

Since in our setting, $b_s\ne0$ is assumed, we have $\dim\mathcal K_s(J, \bb)=s$. Let $X\in\C^{ns\times s}$ be any full-rank matrix whose columns form a basis of $\mathcal K_s(J,\bb)$.
Since this space is $J$-invariant, i.e., $J\mathcal K_s(J,\bb) \subset \mathcal K_s(J,\bb)$, there exists a unique matrix $N\in\C^{s\times s}$ such that 
\begin{equation}\label{eq:X-coordinate}
	JX=XN.
\end{equation}
Moreover, $\bb\in \range(X)$, so there exists $\gamma\in\mathbb C^s$ such that $\bb=X\gamma$. If $X$ has orthonormal columns, then $N=X^*JX$ and $\gamma=X^*\bb$.  Otherwise, we have $N = X^\dagger JX$ and $\gamma=X^\dagger\bb$.

Define
\begin{equation}\label{eq:KX-general}
	\KX :=
	\begin{bmatrix}
	A & \BX\\
	0 & N
	\end{bmatrix},
	\quad
	\BX := EX,
	\quad
	\cX := 
	\begin{bmatrix}
	b_0 \\ \gamma
	\end{bmatrix},
\end{equation}
and
\begin{equation}\label{eq:QX}
	\QX=
	\begin{bmatrix}
		I_n&0\\
		0&X
	\end{bmatrix}.
\end{equation}
Clearly, $\range(\QX) = \C^n\oplus\mathcal K_s(J,\bb)$.
Equation~\eqref{eq:KX-general} represents the fundamental basis-dependent framework. The subscript $\X$ indicates the representation induced by the basis map from the coefficient space to the Krylov subspace associated with the lower-right block
$$
X \colon\C^s\to\mathcal K_s(J,\bb)\subset \C^{ns},\quad 
\gamma\mapsto X\gamma,
$$
where the vector $\gamma$ stores the coefficients of the embedded vector
$X\gamma\in\mathcal K_s(J,\bb)$.  
Accordingly, the one-to-one coordinate map
$$
	\QX: \mathbb C^n\oplus\mathbb C^s\to \mathbb C^n\oplus\mathbb C^{ns}
$$
embeds the \emph{reduced} augmented coefficient space into the larger augmented coefficient space.

A crucial element of our forthcoming analysis is the following relation between Krylov spaces built with $W$ and $\KX$.

\begin{theorem}\label{thm:krylov_W_KX}
Let $W$, $\QX$ and $\KX$ be defined as in~\eqref{eq:aug-mat-blk},~\eqref{eq:KX-general} and~\eqref{eq:QX}, respectively, and let $b = \QX\cX$. Then for any polynomial $p$,
\begin{equation}\label{eq:poly-equivalence}
	p(W)b=\QX\,p(\KX)\cX.
\end{equation}
In particular,
\begin{equation}\label{eq:kiops-krylov-equivalence}
	\mathcal K_m(W,b) = \QX\mathcal K_m(\KX, \cX).
\end{equation}
\end{theorem}
\begin{proof}
A simple calculation shows that
\begin{equation}\label{eq:W-invariant}
	W\QX=
    \begin{bmatrix} A & EX \\ 0 & JX \end{bmatrix}=
    \begin{bmatrix} A & \BX \\ 0 & XN \end{bmatrix}=
    \begin{bmatrix} I_n & \\  & X \end{bmatrix}\begin{bmatrix} A & \BX \\ 0 & N \end{bmatrix}=
    \QX\KX,
\end{equation}
where the second equality follows from~\eqref{eq:X-coordinate}. Repeatedly applying~\eqref{eq:W-invariant} shows
\begin{equation}\label{eq:monomial-equivalence}
W^k\QX = \QX\KX^k, \text{ for all } k\ge 1,
\end{equation}
which directly implies~\eqref{eq:poly-equivalence}. The identity~\eqref{eq:kiops-krylov-equivalence} directly follows from the representation $\mathcal K_m(W,b) = \{p(W)b : p \in\mathcal P_{m-1}\}$.
\end{proof}

\begin{corollary}\label[corollary]{cor:exponential_error}
Let the assumptions of~\Cref{thm:krylov_W_KX} hold and define the first-block projection matrices $\PiK = [I_n,0]\in\C^{n\times(n+s)}$ and $\PiW=[I_n,0]\in\C^{n\times(n+ns)}$.
Then 
\begin{equation}\label{eq:exp-equivalence}
	\e^Wb=\QX\e^{\KX}\cX,
	\quad
	\PiW \e^Wb = \PiK\e^{\KX}\cX,
\end{equation}
and, for any polynomial $p$,
\begin{equation}\label{eq:projected-error-identity}
	\Pi_{K}(\e^{\KX}-p(\KX))\cX =
	\Pi_W(\e^W-p(W))b .
\end{equation}
\end{corollary}

\begin{proof}
We begin by noting that
\begin{equation}\label{eq:piwpik}
\PiW \QX = [I_n, 0] \begin{bmatrix} I_n&0\\0&X \end{bmatrix} = \PiK.
\end{equation}
Using the Taylor series of the exponential together with the relation~\eqref{eq:monomial-equivalence}, we have
$$\e^W\QX =	\sum_{\ell=0}^{\infty}\frac{W^\ell\QX}{\ell!} = \QX\sum_{\ell=0}^{\infty}\frac{\KX^\ell}{\ell!}	= \QX\e^{\KX}.$$
The first equality in~\eqref{eq:exp-equivalence} follows by using $b=\QX\cX$, and the second one then follows from the first one via~\eqref{eq:piwpik}. Equation~\eqref{eq:projected-error-identity} now follows from~\eqref{eq:exp-equivalence} and the fact that~\eqref{eq:poly-equivalence} together with~\eqref{eq:piwpik} implies $\PiW p(W)b= \PiK p(\KX) \cX$.
\end{proof}

The identities~\eqref{eq:poly-equivalence} and~\eqref{eq:exp-equivalence} imply that the desired vector in~\eqref{eq:phi-lin-comb} can be obtained by evaluating the action of the exponential for the reduced augmented matrix $\KX$ on the vector $\cX$ defined in~\eqref{eq:KX-general}, and the resulting projected polynomial error is exactly the projected error of the larger block formulation~\eqref{eq:aug-mat-blk} for the same polynomial. 

The Krylov subspace generated by the reduced pair, $\mathcal K_m(\KX,\cX)\subset\C^n\oplus\C^s$, is embedded in the larger augmented space as $\QX\mathcal K_m(\KX,\cX)\subset \C^n\oplus\C^{ns}$.
Different choices of $X$ only change the coordinate representations $\BX$, $N$, and $\gamma$ of~\eqref{eq:KX-general}, but preserve the projected actions~\eqref{eq:poly-equivalence} and~\eqref{eq:exp-equivalence}.

\subsection{The basis choice of KIOPS}

Define 
\begin{equation}\label{eq:XK}
	\XK=
	\begin{bmatrix}
	J^{s-1}\bb & J^{s-2}\bb & \cdots & J\bb & \bb
	\end{bmatrix}
	= 
	\begin{bmatrix}
		b_s & b_{s-1} & \ldots & \ldots & b_1 \\
		0   & b_s & \ddots  & & b_2 \\
		\vdots & \ddots & \ddots & \ddots & \vdots\\
        \vdots & & \ddots & b_s & b_{s-1} \\
		0 & \ldots & \ldots & 0 & b_s
	\end{bmatrix} \in \C^{ns \times s},
\end{equation}
which has a block upper triangular Toeplitz structure.
By~\Cref{lem:lb-kry-dim}, the columns of $\XK$ form a
nonorthogonal Jordan-chain basis of $\mathcal K_s(J,\bb)$. Moreover, for every $v=\XK\gamma\in\mathcal K_s(J,\bb)$, the coefficient vector of $Jv$ with respect to this basis is $J_s(0)\gamma$. Specifically,
$$
	J\XK=\XK J_s(0),\quad
	E\XK=[b_s,b_{s-1},\ldots,b_1],
	\quad
	\bb=\XK e_s,
$$
and so the coordinate matrix in~\eqref{eq:KX-general} becomes the augmented matrix $K$ of KIOPS in~\eqref{eq:aug-mat-kiops}. The connection with the larger block formulation in~\eqref{eq:aug-mat-blk} is
\begin{equation}\label{eq:kiops-embedding}
	W\QK=\QK K,
	\quad
	b=\QK c= \QK\begin{bmatrix}
		b_0\\ e_s
	\end{bmatrix},
	\quad
	\QK=
	\begin{bmatrix}
	I_n&0\\
	0&\XK
	\end{bmatrix},
\end{equation}
where $c$ is defined in~\eqref{eq:poly-krylov-KIOPS}.

As we have discussed above, different choices of basis $X$ change the pair $(\BX,N)$ in~\eqref{eq:KX-general}.  An orthonormal basis gives a bounded coupling $\normt{\BX}\le1$, but generally produces a dense nilpotent matrix $N$.
The KIOPS basis $\XK$ yields the standard nilpotent Jordan
representation $G=J_s(0)$ and the simple coupling $\BX=[b_s,\ldots,b_1]$, at the cost of a nonorthogonal embedding $\QK$.

The relations derived in~\Cref{thm:krylov_W_KX} and~\Cref{cor:exponential_error} remain applicable for the specific choice of KIOPS basis $\XK$.
Note that we distinguish quantities associated with the KIOPS basis from those associated with a general basis by using the subscripts $\rm K$ and $\rm X$, respectively.

\subsubsection{Conditioning of the KIOPS basis}\label{subsubsec:kiops-cond}

The block triangular Toeplitz structure in the $\XK$ of~\eqref{eq:XK} immediately implies that it has full column rank if and only if $b_s\ne0$, which is our standard assumption that we have made based on~\eqref{eq:phi-lin-comb}.
This KIOPS basis matrix $\XK$ can therefore be ill-conditioned if the diagonal vector $b_s$ is small relative to the other vectors $b_j$.
But what more can be said about how this conditioning depends on the right-hand-side vectors $b_j$? 
To this end, we look at the Gram matrix of $\XK$, which, by the block upper triangular Toeplitz structure, is given explicitly as
\begin{equation}\label{eq:XK-gram}
	\GK := \XK^*\XK,
	\quad
	(\GK)_{ij}
	=
	\sum_{\ell=1}^{\min\{i,j\}}
	b_{s-i+\ell}^* b_{s-j+\ell},
	\quad 1\le i,j\le s,
\end{equation}
where the $(i,j)$ entry is obtained by taking the inner product of columns $i$ and $j$ of $\XK$ and then summing over their common nonzero block rows.

Now suppose that $b_1=b_2=\cdots=b_s=v\ne0$.\footnote{Of course, this is an academic example, as in the case $b_1 = \dots = b_s = v$, the action of all $\varphi$-functions could be approximated in the same (standard) Krylov space $\mathcal{K}_s(A,v)$, so that there is no need for augmenting the matrix at all.} We have, from~\eqref{eq:XK-gram},
\[
    \GK
    =
    \normt{v}^2 M_s,
    \quad
    (M_s)_{ij}=\min\{i,j\},\quad 1\le i,j\le s,
\]
where $M_s$ is the \texttt{minij} matrix from the MATLAB gallery.
It is known that $M_s$ has a tridiagonal inverse, and that the eigenvalues are given by
\[
    \lambda_k(M_s)
    =
    \frac{1}{2-2\cos((2k-1)\pi/(2s+1))},
    \quad k=1,\ldots,s,
\]
so that $\kappa_2(M_s)\sim 16s^2/\pi^2$; see~\cite{ruth52}. It follows that the condition number of $\XK$ grows linearly in $s$,
\[
\kappa_2(\XK) = \sqrt{\kappa_2(\GK)} = \sqrt{\kappa_2(M_s)}\sim 4s/\pi.
\]

It appears from the above example that the conditioning of $\XK$ has little to do with the collinearity of the vectors $b_j$.
On the other hand, a common scaling $b_j\mapsto \beta b_j$ for all $j$, with $\beta\ne 0$, replaces $\XK$ by $\beta\XK$ and hence will leave $\kappa_2(\XK)$ unchanged, since each block entry of $\XK$ depends linearly on the $b_j$.

\section{Basis-induced metric}\label{sect:basis-metric}

Returning to the general basis-dependent framework of~\eqref{eq:KX-general}, the identity $W\QX=\QX\KX$ of~\eqref{eq:W-invariant} shows that $\KX$ and $W$ describe the same action after embedding. To be precise, for every reduced coordinate vector $u=[\nu^*,\gamma^*]^*\in\C^n\oplus\C^s$, applying $\KX$ to $u$ and then embedding with $\QX$ gives the same vector as first embedding $u$ into $\C^n\oplus\C^{ns}$ and then applying $W$.
However, the embedding $\QX$ of~\eqref{eq:QX} does not generally preserve Euclidean distance.
Indeed, it does so if and only if $X^*X=I$, since
\begin{equation}\label{eq:metric-norm-Qx} 
    \normt{\QX u}^2=\normt{\nu}^2+\normt{X\gamma}^2,
    \quad
    \normt{u}^2=\normt{\nu}^2+\normt{\gamma}^2.
\end{equation}

We therefore equip the reduced coordinate space $\HX := \C^n\oplus\C^s$ with the inner product induced by the embedding $\QX$,
\begin{equation}\label{eq:inner-product-induced}
	\langle z, u \rangle_{\MX} := (\QX z)^* (\QX u) = z^*\MX u,
	\quad \MX := \QX^* \QX.
\end{equation}
Since $\QX$ has full column rank, $\MX$ is Hermitian positive definite,
and hence $\langle\cdot,\cdot\rangle_{\MX}$ defines an inner product on $\HX$.
The corresponding $\MX$-induced norm is
\begin{equation}\label{eq:M-X-norm}
	\norm{z}_{\MX}^2 := z^*\MX z = \normt{\QX z}^2,
\end{equation}
which is effectively the Euclidean norm of the embedded vector in the larger block space.
In particular, since $\PiK z=\PiW \QX z$ and $\normt{\PiW}=1$ for the projections defined in~\Cref{cor:exponential_error}, we have
\begin{equation}\label{eq:projection-metric-norm}
	\normt{\Pi_{K}z} = \normt{\Pi_W\QX z} \le \normt{\QX z} =\norm{z}_{\MX}.
\end{equation}

Finally, since $\MX$ is Hermitian positive definite, its principal square root exists and is itself Hermitian positive definite, and hence nonsingular~\cite[Chap.~6]{high:FM}. We denote the principal square root by
\begin{equation}\label{eq:M-sqrt}
    \RX := \MX^{1/2}.
\end{equation}
In particular, $\MX=\RX^2=\RX^*\RX$.
It follows, analogously to~\eqref{eq:M-X-norm}, that $\norm{z}_{\MX}=\normt{\RX z}$.
Therefore, the induced matrix norm satisfies
$$
	\norm{A}_{\MX} :=
	\max_{z\ne0} \frac{\norm{Az}_{\MX}} {\norm{z}_{\MX}} =
	\max_{z\ne0} \frac{\normt{\RX A z}}{\normt{\RX z}}.
$$
With the change of variables $u= \RX z$, this gives
\begin{equation}\label{eq:M-X-norm-mat}
\norm{A}_{\MX} =
\max_{u \ne0} \frac{\normt{\RX A \RX^{-1} u}}{\normt{u}}
= \normt{\RX A \RX^{-1}},
\end{equation}
which identifies the $\MX$-induced matrix norm with the Euclidean norm of the similar matrix $\RX A\RX^{-1}$.
Note that $\QX$ is generally rectangular and therefore cannot be used in place of $\RX$ in~\eqref{eq:M-X-norm-mat}.

\subsection{Metric distortion}\label{subsect:metric-distortion}
Having defined the $\MX$-induced metric in~\eqref{eq:M-X-norm}, it is natural to ask how much it can deviate from the Euclidean metric on the reduced coordinates, and under what circumstances this distortion becomes substantial.
For $u=[\nu^*,\gamma^*]^*$, subtracting the two expressions in~\eqref{eq:metric-norm-Qx} gives
\begin{equation}\label{eq:metric-norm-diff-pre} 
\norm{u}_{\MX}^2 - \normt{u}^2 = \normt{X\gamma}^2 - \normt{\gamma}^2,
\end{equation}
which shows that the squared difference between the induced and Euclidean squared norms is governed by the extent to which the basis matrix $X$ fails to be an isometry.  
More generally, the following result characterizes the metric distortion, in both additive and multiplicative forms, in terms of the singular values of $X$.

\begin{theorem}\label{thm:basis-metric-distortion}
For any full-rank matrix $X\in\C^{ns\times s}$ whose columns form a basis of $\mathcal K_s(J,\bb)$ and $u=[\nu^*,\gamma^*]^*\in\C^n\oplus\C^s$, the $\MX$-induced norm~\eqref{eq:M-X-norm} and the Euclidean norm satisfy
\begin{equation}\label{eq:metric-norm-diff}
	\sup_{u\ne0}
	\frac{\bigl|\norm{u}_{\MX}^2-\normt{u}^2\bigr|}
	     {\normt{u}^2}
	= \max_i|\sigma_i(X)^2-1|,
\end{equation}
where the supremum is attained by vectors with $\nu=0$. Moreover,
\begin{equation}\label{eq:metric-norm-ratio}
		\sup_{u\ne0}\frac{\norm{u}_{\MX}}{\normt{u}}
		= \max\{1,\sigma_{\max}(X)\}, 
		\quad
		\sup_{u\ne0}\frac{\normt{u}}{\norm{u}_{\MX}}
		= \max\{1, 1/\sigma_{\min}(X)\}.
\end{equation}
\end{theorem}

\begin{proof}
Define $G=X^*\!X$. Then $\norm{u}_{\MX}^2-\normt{u}^2 = \gamma^*(G-I)\gamma$ by~\eqref{eq:metric-norm-diff-pre}. Hence
\[
	\frac{\bigl|\norm{u}_{\MX}^2-\normt{u}^2\bigr|}{\normt{u}^2}
	=
	\frac{|\gamma^*(G-I)\gamma|}{\normt{\nu}^2+\normt{\gamma}^2}.
\]
For each fixed $\gamma$, the numerator is independent of $\nu$, while the denominator is minimized when $\nu=0$. Therefore
\begin{equation}\label{eq:metric-norm-sup-gamma}
\begin{aligned}
	\sup_{u\ne0} \frac{\bigl|\norm{u}_{\MX}^2-\normt{u}^2\bigr|}
	{\normt{u}^2}
	&=
	\sup_{\gamma\ne0}\frac{|\gamma^*(G-I)\gamma|}{\normt{\gamma}^2}.
\end{aligned}
\end{equation}
Since $G-I$ is Hermitian, the Rayleigh--Ritz characterization \cite[sect.~4.2]{hojo13} gives
\begin{equation}\label{eq:metric-norm-bound-gamma}
    \sup_{\gamma\ne0}
    \frac{\bigl|\gamma^*(G-I)\gamma\bigr|}{\gamma^*\gamma} 
    = 
    \normt{G-I}
    = 
    \max_i |\sigma_i(X)^2-1|,
\end{equation}
where the last equality follows because the eigenvalues of $G-I$ are $\sigma_i(X)^2-1$.
Combining~\eqref{eq:metric-norm-sup-gamma} and~\eqref{eq:metric-norm-bound-gamma} proves~\eqref{eq:metric-norm-diff}. 
The supremum is attained by taking $\gamma$ to be an eigenvector of $G-I$ corresponding to an eigenvalue of largest absolute value and setting $\nu=0$.

Finally, recall from~\eqref{eq:QX} and~\eqref{eq:inner-product-induced} that $\MX$ is block diagonal with eigenvalues $1$ and $\sigma_i(X)^2$.
For the first ratio in~\eqref{eq:metric-norm-ratio},
\begin{equation*}
    \left(\sup_{u\ne0}\frac{\norm{u}_{\MX}}{\normt{u}}\right)^2
	=
	\sup_{u\ne0}\frac{u^*\MX u}{u^*u}
	=
	\lambda_{\max}(\MX) 
	=
	\max\{1,\sigma_{\max}(X)^2\}.
\end{equation*}
For the second ratio, taking $v=\MX^{1/2}u$ gives
\begin{equation*}
	\left(\sup_{u\ne0}\frac{\normt{u}}{\norm{u}_{\MX}}\right)^2
	=
	\sup_{v\ne0}\frac{v^*\MX^{-1}v}{v^*v} \\
	=
	\lambda_{\max}(\MX^{-1})
	=
	\max\{1,1/\sigma_{\min}(X)^2\}.
\end{equation*}
Taking square roots proves~\eqref{eq:metric-norm-ratio}.    
\end{proof}

\Cref{thm:basis-metric-distortion} shows that induced and Euclidean metrics are close when the singular values of $X$ are clustered near one, whereas some singular values of $X$ being far from one lead to a corresponding distortion of lengths in the reduced coordinates.

\subsubsection{Sensitivity of KIOPS basis to right-hand-side scaling}
For the KIOPS basis $X=\XK$ of~\eqref{eq:XK}, the special structure enables a more explicit analysis of how the metric distortion depends on the right-hand-side scaling.
A common scaling $b_j\mapsto \beta b_j$ with $\beta\ne 0$ replaces $\XK$ by $\beta\XK$ and hence maps every singular value from $\sigma_i(\XK)$ to $\beta\sigma_i(\XK)$.
The following~\Cref{cor:rhs-scaling-metric-distortion} describes how this scaling can change the metric distortion relative to the fixed Euclidean metric.

\begin{corollary}\label[corollary]{cor:rhs-scaling-metric-distortion}
Let $b_1,\ldots,b_s$ be a family of right-hand-side vectors, and let $\XK$ be the corresponding KIOPS basis matrix of~\eqref{eq:XK}.
For the scaled right-hand-side vectors $\overline b_j=\beta b_j$, $\beta>0$ we can define the corresponding $\beta$-dependent induced norm by
\[
    \norm{u}_{\overline \MK(\beta)}^2 := u^*\overline \MK(\beta) u,
    \quad 
    \overline \MK(\beta)=
    \begin{bmatrix} I_n & \\  & \beta^2\XK^*\XK \end{bmatrix}.
\]
Then
\begin{equation}\label{eq:metric-norm-diff-beta}
    \sup_{u\ne0}
    \frac{\bigl|\norm{u}_{\overline \MK(\beta)}^2-\normt{u}^2\bigr|}
    {\normt{u}^2}
    =
    \max_i|\beta^2\sigma_i(\XK)^2-1|
\end{equation}
and
\begin{equation}\label{eq:metric-norm-ratio-beta}
    \begin{split}
    \sup_{u\ne0}
    \frac{\norm{u}_{\overline \MK(\beta)}}{\normt{u}}
    &= \max\{1, \beta\sigma_{\max}(\XK)\}, \\
    \sup_{u\ne0}
    \frac{\normt{u}}{\norm{u}_{\overline \MK(\beta)}}
    &= \max\{1, 1/(\beta\sigma_{\min}(\XK))\}.
    \end{split}
\end{equation}
\end{corollary}
\begin{proof}
Define the KIOPS basis matrix corresponding to the scaled vectors $\overline b_j$ as $\overline \XK(\beta):=\beta\XK$, and set the Gram matrix to be $\overline \GK(\beta) := \overline \XK(\beta)^*\overline \XK(\beta) = \beta^2\GK$, recalling from~\eqref{eq:XK-gram} that $\GK=\XK^*\XK$. 
Therefore, we can define the scaled embedding matrix and the corresponding metric matrix as
\[
    \overline \QK(\beta)=
	\begin{bmatrix}
	I_n&0\\
	0& \overline \XK(\beta)
	\end{bmatrix},
    \quad
    \overline \MK(\beta) = \overline \QK(\beta)^* \overline \QK(\beta) = 
    \begin{bmatrix}
	I_n&0\\
	0& \beta^2\XK^*\XK
	\end{bmatrix}.
\]
The expressions~\eqref{eq:metric-norm-diff-beta} and~\eqref{eq:metric-norm-ratio-beta} are obtained by applying~\Cref{thm:basis-metric-distortion} to the scaled basis $\overline \XK(\beta)$ and noting that $\sigma_i(\overline \XK(\beta))=\beta\sigma_i(\XK)$ for every $i$.
\end{proof}

The scaling effect on the metric distortion discussed in this subsection depends on the locations of the singular values relative to one. It should therefore not be conflated with the effect on the conditioning $\kappa_2(\XK)$, which measures only the spread of the singular values. As discussed in~\Cref{subsubsec:kiops-cond}, a common scaling $b_j\mapsto \beta b_j$ leaves $\kappa_2(\XK)$ unchanged.

\section{Field-of-values equivalence} \label{sect:fov}
As outlined in \Cref{sect:poly-krylov}, the field of values plays a crucial role in understanding the convergence behavior of Krylov subspace methods; see in particular~\eqref{eq:krylov_convergence_bound}. We therefore now investigate how the fields of values of the different operators under consideration are related to each other.

To do so, we begin by presenting a simple, auxiliary result on the field of values of block upper triangular matrices.

\begin{proposition}\label[proposition]{pro:fov_triangular}
Let 
$$
M =
\begin{bmatrix}
	A & C\\
	0 & D
\end{bmatrix}.
$$
Then 
\begin{equation}\label{eq:fov-inclusion}
\F(M) \subseteq \conv\bigl(\F(A)\cup\F(D)\bigr) + \{\zeta\in\C: |\zeta|\le \tfrac12\normt{C}\},
\end{equation}
where $\conv(\cdot)$ denotes the convex hull of a set.
\end{proposition}
\begin{proof}
Let $z=[x^*,w^*]^*$ be a unit vector, with sizes of $x$ and $w$ consistent with the block partitioning of $M$. We then have 
$$ \F(M)=\{x^*Ax+w^*Dw+x^*Cw \colon \normt{x}^2+\normt{w}^2=1\}. $$
If $x,w\ne0$, then 
$$ x^*Ax+w^*Dw = \normt{x}^2 \frac{x^*Ax}{\normt{x}^2} + \normt{w}^2 \frac{w^*Dw}{\normt{w}^2},$$
where
$$\frac{x^*Ax}{\normt{x}^2}\in\F(A), \quad \frac{w^*Dw}{\normt{w}^2}\in\F(D).$$
This shows $x^*Ax+w^*Dw$ form a convex combination of points in $\F(A)$ and $\F(D)$. The cases $x=0$ or $w=0$ are also included by the same argument. Moreover,
$$
|x^*Cw| \le \normt{C}\normt{x}\normt{w} \le \frac{1}{2} (\normt{x}^2+\normt{w}^2)\normt{C}
= \frac{1}{2}\normt{C},
$$
which yields the inclusion~\eqref{eq:fov-inclusion}.
\end{proof}

 Applying~\Cref{pro:fov_triangular} to the full block matrix $W$ of~\eqref{eq:aug-mat-blk} gives
\begin{equation}\label{eq:fov-W}
	\F(W)
	\subseteq
	\conv\bigl(\F(A)\cup\F(J)\bigr)
	+
	\{\zeta\in\C\colon |\zeta|\le \tfrac{1}{2}\},
\end{equation}
because $\normt{E}=1$.  On the other hand, applying the same inclusion to the KIOPS matrix of~\eqref{eq:aug-mat-kiops} yields
\begin{equation}\label{eq:fov-K}
	\F(K)
	\subseteq
	\conv\bigl(\F(A)\cup\F(J_s(0))\bigr)
	+
	\{\zeta\in\C: |\zeta|\le \tfrac12\normt{[b_s,\ldots,b_1]}\}.
\end{equation}
Since $J=J_s(0)\otimes I_n$ is unitarily similar to $I_n\otimes J_s(0)=\diag(J_s(0),\ldots,J_s(0))$ via a permutation matrix (see~\cite[sect.~B.13]{high:FM}), we have
$$
	\F(J) = \F\bigl(\diag(J_s(0),\ldots,J_s(0))\bigr) = \conv\bigl(\F(J_s(0))\bigr)
	= \F(J_s(0)).
$$
This field of values is explicitly given by 
\begin{equation}\label{eq:fov-J}
    \F(J) = \F(J_s(0)) = \{\zeta\in\C\colon |\zeta|\le \cos(\tfrac{\pi}{s+1})\},
\end{equation}
the closed disk centered at the origin with radius $\cos(\tfrac{\pi}{s+1})$; see, e.g., \cite[Prop.~1]{hade92}.

Unlike~\eqref{eq:fov-W}, the bound for $\F(K)$ involves
$\normt{[b_s,\ldots,b_1]}$ and is therefore not uniform in the right-hand
side data.  In comparison,~\eqref{eq:fov-W} enlarges
$\conv(\F(A)\cup\F(J))$ only by the fixed disk of radius $1/2$.
This contrast can be interpreted using the metric introduced in~\Cref{sect:basis-metric}, which we now use to compare the field of values of the reduced coordinate matrix with that of the full block formulation.

\subsection{Metric field of values induced by the basis}
The induced inner product~\eqref{eq:inner-product-induced} motivates the metric field of values~\cite{give52, gura97} 
\begin{equation}\label{eq:metric-fov}
\F_{\MX}(\KX)
=
\left\{
\frac{\langle z, \KX z \rangle_{\MX}}{\langle z, z \rangle_{\MX}} \colon z\ne0
\right\},
\quad 
\MX=\QX^*\QX=
\begin{bmatrix}
	I_n&0\\
	0&X^*X
\end{bmatrix}.
\end{equation}
When $X$ has orthonormal columns, $\MX=I$ and the metric field of values agrees with the standard Euclidean field of values of $\KX$. 

The following result characterizes $\F_{\MX}(\KX)$ as the Euclidean numerical range of a matrix that is similar to $\KX$ and furthermore shows that it is always contained in $\F(W)$.

\begin{lemma}\label[lemma]{lem:metric_fov_kx}
The metric field of values satisfies
\begin{equation}\label{eq:fov-similar-equiv}
\F_{\MX}(\KX) = \F(R_X\KX R_X^{-1})
\end{equation}
where $\RX$ is defined in~\eqref{eq:M-sqrt}. Furthermore,
\begin{equation}\label{eq:fov-bound}
\F_{\MX}(\KX) \subseteq	\F(W).
\end{equation}
\end{lemma}
\begin{proof}
Let $u=\RX z$. Then $z\ne0$ if and only if $u\ne0$, and
\[
    z^*\MX z = z^*\RX^*\RX z = u^*u,
\]
while
\[
    z^*\MX\KX z = z^*\RX^*\RX\KX z = u^*\RX\KX\RX^{-1}u.
\]
Hence, equation~\eqref{eq:fov-similar-equiv} follows via
\begin{equation}\label{eq:fov-similar-equiv-intermediate}
\F_{\MX}(\KX)
    =
    \left\{
    \frac{z^*\MX\KX z}{z^*\MX z}:z\ne0
    \right\}  \\
    =
    \left\{
    \frac{u^*\RX\KX\RX^{-1}u}{u^*u}:u\ne0
    \right\}
    =
    \F(\RX\KX\RX^{-1}).
\end{equation}

Substituting~\eqref{eq:W-invariant} into the first equality in~\eqref{eq:fov-similar-equiv-intermediate} and rewriting $\MX$ as $\QX^*\QX$ gives
\begin{equation*}
	\F_{\MX}(\KX)
	=
	\left\{
	\frac{(\QX z)^*W(\QX z)}{\normt{\QX z}^2}: z\ne0
	\right\}
	\subseteq
	\F(W),
\end{equation*}
where the inclusion uses the full column rank of $\QX$, so that $\QX z\ne0$ whenever $z\ne0$. This proves~\eqref{eq:fov-bound} and thus concludes the proof.
\end{proof}

\subsubsection{KIOPS specialization}\label{subsubsec:kiops-metric-fov}
We now apply the preceding metric field-of-values comparison to the KIOPS basis matrix $\XK$.
For this choice, the general coordinate matrix $\KX$ in~\eqref{eq:KX-general} is the KIOPS matrix $K$ in~\eqref{eq:aug-mat-kiops}, and the embedding is $\QK$ from~\eqref{eq:kiops-embedding}.
Thus $\MK=\QK^*\QK$, and the corresponding metric field of values is
\begin{equation}\label{eq:fov-MK-K}
	\F_{\MK}(K) = \F(R_{\rm K}K R_{\rm K}^{-1}),\quad
	R_{\rm K}^*R_{\rm K}=\MK=\begin{bmatrix}
		I_n&0\\
		0&\XK^*\XK
	\end{bmatrix},
\end{equation}
for the principal matrix square root $R_{\rm K}=\MK^{1/2}$; cf.~\eqref{eq:M-sqrt}.
Taking $X=\XK$ in~\eqref{eq:fov-bound} gives
\[
	\F_{\MK}(K)\subseteq \F(W).
\]
Consequently, the metric field of values of the KIOPS coordinate matrix is controlled by the block-matrix bound~\eqref{eq:fov-W}.
This is the appropriate comparison with the embedded formulation: the field of values is invariant under unitary similarity, but not under a general non-unitary similarity. Therefore, the Euclidean field of values of the coordinate matrix $K$ is not expected to match that of $W$.
Indeed, the elementary estimate~\eqref{eq:fov-K} shows that $\F(K)$ may grow with the size of the right-hand-side block $B = [b_s,\ldots,b_1]$.
The metric field of values $\F_{\MK}(K)$ removes this coordinate distortion by measuring a reduced vector $z$ through the embedded norm $\normt{\QK z}=\norm{z}_{\MK}$.
As analyzed in~\Cref{subsect:metric-distortion}, the difference between this embedded metric and the Euclidean metric is governed by the singular values of the basis matrix, here $\XK$.

The same viewpoint explains the effect of uniform right-hand-side scaling.
Let $\overline b_j=\beta b_j$, $\overline B=\beta B$, and
\[
    \overline K(\beta)
    =
    \begin{bmatrix}
        A&\overline B\\
        0&J_s(0)
    \end{bmatrix}.
\]
The Euclidean field of values $\F\bigl(\overline K(\beta)\bigr)$ can become dominated by the scaled off-diagonal block $\overline B$, which is the behavior anticipated by~\eqref{eq:fov-K}.
In contrast, let $\GK=\XK^*\XK$ be the Gram matrix in~\eqref{eq:XK-gram}, and let its principal square root be $C_{\rm K}:=\GK^{1/2}$.
By~\Cref{cor:rhs-scaling-metric-distortion}, the scaled metric is
$\overline \MK(\beta)=\diag(I_n,\beta^2\GK)$ and we have
\[
    \overline R_{\rm K}(\beta) := \diag(I_n,\beta C_{\rm K}) 
    = \overline \MK(\beta)^{1/2}.
\]
Therefore, by~\Cref{lem:metric_fov_kx},
\begin{equation}\label{eq:fov-MK-K-beta}    
    \F_{\overline \MK(\beta)}(\overline K(\beta))
    =
    \F\!\left(
    \overline R_{\rm K}(\beta)
    \overline K(\beta)
    \overline R_{\rm K}(\beta)^{-1}
    \right)
    =
    \F\!\left(
    \begin{bmatrix}
        A&B C_{\rm K}^{-1}\\
        0&C_{\rm K} J_s(0)C_{\rm K}^{-1}
    \end{bmatrix}
    \right),
\end{equation}
which is independent of $\beta$. Hence, uniform right-hand-side scaling may change the Euclidean field of values of the scaled KIOPS matrix, whereas the corresponding scaled metric field of values stays invariant.

\subsection{Polynomial approximation bounds}
As outlined in \Cref{sect:poly-krylov}, best uniform polynomial approximations play an important role in studying the convergence behavior of Krylov methods. Thus, the inclusion~\eqref{eq:fov-bound} is useful for understanding the error of Krylov approximations for  the augmented matrix formulations. 

The following result uses the above ingredients to bound the norm of the projected error~\eqref{eq:projected-error-identity} in terms of best polynomial approximation on the metric field of values.

\begin{theorem}\label{eq:poly_approx_metric_fov}
We have
\begin{equation}\label{eq:projected-error}
	\inf_{p\in\mathcal P_{m-1}} \normt{\Pi_{\KX}(\e^{\KX}-p(\KX))\cX} \leq C_{\rm cr} \normt{b} E_{m-1}\bigl(\e^z,\F(W)\bigr),
\end{equation}
\end{theorem}
\begin{proof}
Applying the bound~\eqref{eq:crouzeix-ineq} to $\RX \KX \RX^{-1}$ gives, for every function $f$ analytic on a neighborhood of $\F(\RX \KX \RX^{-1})$,
$$
	\normt{f(\RX \KX \RX^{-1})} \le C_{\rm cr} \max_{\zeta\in\F(\RX \KX \RX^{-1})} |f(\zeta)|,
$$
where, by~\eqref{eq:M-X-norm-mat}, the left-hand side satisfies
$$
	\normt{f( \RX \KX \RX^{-1} )} = \normt{\RX f(\KX) \RX^{-1}} = \norm{f(\KX)}_{\MX}.
$$
Using~\eqref{eq:fov-similar-equiv}, we obtain
\begin{equation}\label{eq:crouzeix-ineq-MX}
	\norm{f(\KX)}_{\MX} \le C_{\rm cr} \max_{\zeta\in\F_{\MX}(\KX)} |f(\zeta)|,
\end{equation}
which is, in fact, the Crouzeix--Palencia bound~\eqref{eq:crouzeix-ineq} applied in the finite-dimensional Hilbert space $\mathcal \HX$ defined in~\eqref{eq:inner-product-induced}; see~\cite{crpa17}. Taking $f(\zeta)=\e^\zeta-p(\zeta)$ in~\eqref{eq:crouzeix-ineq-MX}, with $p\in\mathcal P_{m-1}$, yields
\begin{equation}\label{eq:metric-function-bound}
	\norm{(\e^{\KX}-p(\KX)) \cX}_{\MX} \le
	C_{\rm cr} \norm{\cX}_{\MX}
	\max_{\zeta\in\F_{M_{\rm K}}(\KX)}
	|\e^\zeta-p(\zeta)|.
\end{equation}
Substituting the bound~\eqref{eq:projection-metric-norm} into the left-hand side of~\eqref{eq:metric-function-bound} and using $\norm{\cX}_{\MX} = \normt{\QX \cX} = \normt{b}$, we obtain, for every $p\in\mathcal P_{m-1}$,
$$
	\normt{\Pi_{\KX}(\e^{\KX}-p(\KX))\cX} \le
	C_{\rm cr} \normt{b} \max_{\zeta\in\F_{\MX}(\KX)} |\e^\zeta-p(\zeta)|.
$$
Taking the infimum over $p\in\mathcal P_{m-1}$ now gives
\begin{align*}
	\inf_{p\in\mathcal P_{m-1}} \normt{\Pi_{\KX}(\e^{\KX}-p(\KX))\cX} 
	& \le
	C_{\rm cr} \normt{b} E_{m-1} \bigl(\e^z,\F_{\MX}(\KX)\bigr) \\ 
	& \le
	C_{\rm cr} \normt{b} E_{m-1}\bigl(\e^z,\F(W)\bigr),
\end{align*}
where the second inequality follows from $\F_{\MX}(\KX)\subseteq\F(W)$ in~\eqref{eq:fov-bound}.
\end{proof}

The following corollary specializes the above result to the basis choice used in KIOPS.

\begin{corollary}\label[corollary]{cor:poly_error_kiops}
For the approximant~\eqref{eq:poly-krylov-KIOPS} produced by KIOPS (in the usual Euclidean space), we have the error bound
\begin{equation}\label{eq:projected-error-kiops}
	\inf_{p\in\mathcal P_{m-1}} \normt{\Pi_{K}(\e^{K}-p(K))c}  \le
	C_{\rm cr} \normt{b} E_{m-1}\bigl(\e^z\F(W)\bigr).
\end{equation}    
\end{corollary}
\begin{proof}
The result immediately follows by using $\XK$ as the basis $X$ in~\eqref{eq:projected-error}.
\end{proof}
Since $\F(W)$ is bounded as in~\eqref{eq:fov-W}, the approximation region in~\eqref{eq:projected-error-kiops} is independent of the individual norms of $b_1,\ldots,b_s$.  These vectors appear only through the scale factor $\normt{b}$. The bound~\eqref{eq:projected-error-kiops} is the basic mechanism by which field-of-values information for the larger block formulation~\eqref{eq:aug-mat-blk} can be used for the KIOPS formulation~\eqref{eq:aug-mat-kiops}. 

Standard techniques for bounding the best uniform polynomial approximation error $E_{m-1}\bigl(\e^z,\F(W)\bigr)$ involve conformal mappings and the Faber transform; see, e.g.,~\cite{bere09}. However, due to the non-trivial shape of the field of values of the augmented matrix, the respective conformal mappings are typically not known in closed form. Therefore, explicitly evaluating~\eqref{eq:projected-error-kiops} will in general involve numerically computing a conformal mapping, or estimating the polynomial approximation error via, e.g., a Remez-type algorithm or an approach as described in \Cref{appendix:poly_error}. Even in the arguably simplest case that $A$ is Hermitian negative semidefinite, so that $\F(A)$ reduces to an interval on the negative real line, it is impossible to explicitly construct a conformal mapping for the enclosing set on the right-hand side of~\eqref{eq:fov-W}. 

The following auxiliary result demonstrates how to obtain an upper bound for the error by further enclosing this set in a Bernstein ellipse and using a truncated Chebyshev series for estimating the error. Note that one cannot expect this bound to be sharp, as the actual spectral region is replaced by a slightly larger set. However, as we will demonstrate, it at least yields an error bound that is much more reminiscent of the actual behavior than what one would obtain from a bound on $\F(K)$.

\begin{proposition}\label[proposition]{pro:enclosing_ellipse}
Let $A$ be negative semidefinite with $\F(A)=[-\alpha, 0]$, $\alpha>0$. Then 
\[
    E_{m-1}(\e^z,\F(W))
    \leq
    \e^{-h}
    \sum_{k=m}^{\infty}
    B_k(h)\left(\chi_s^k+\chi_s^{-k}\right),
\]
where $B_k$ denotes the modified Bessel function of the first kind and
\[
    \chi_s := 1+\frac{R_s}{h} + \sqrt{\frac{R_s}{h}\left(2+\frac{R_s}{h}\right)},
    \quad
    h:=\frac{\alpha}{2},
    \quad
    R_s:=r_s+\frac12,
    \quad
    r_s=\cos\left(\frac{\pi}{s+1}\right).
\]
\end{proposition}

\begin{proof}
We recall from~\eqref{eq:fov-J} that $\F(J) = \{\zeta\in\C\colon |\zeta| \leq r_s\}$, so that $\conv\bigl([-\alpha,0]\cup \F(J)\bigr) = \conv\bigl(\{-\alpha\}\cup  \{\zeta\in\C\colon |\zeta| \leq r_s\}\bigr)$. Moreover, by standard properties of the Minkowski sum, we have
\[
\conv\bigl(\{-\alpha\}\cup  \{\zeta\in\C\colon |\zeta| \leq r_s\}\bigr) \subseteq [-\alpha,0]+\{\zeta\in\C\colon |\zeta| \leq R_s\} =: \mathcal{S}_s.
\]
We now further enclose the Bunimovich stadium $\mathcal{S}_s$ in a Bernstein ellipse $\mathcal{E}_s$ with foci at $-\alpha$ and $0$ and semi-axes
\[
a_s:=h+R_s, \quad b_s:=\sqrt{a_s^2-h^2}=\sqrt{R_s(2h+R_s)},
\]
i.e., we take
\[
\mathcal{E}_s = \left\{ -h+\frac{h}{2}\left(\zeta+\zeta^{-1}\right) : |\zeta|\leq \chi_s \right\},
\]
where
\[
    \chi_s=\frac{a_s+b_s}{h} = \frac{h+R_s+\sqrt{R_s(2h+R_s)}}{h} = 
    1 + \frac{R_s}{h} + \sqrt{\frac{R_s}{h}\left(2+\frac{R_s}{h}\right)}.
\]
We now estimate the polynomial approximation error on $\mathcal{E}_s$. Note that via the affine change of variables
\[
    \zeta=-h+h\xi, \quad \xi=\frac{\zeta+h}{h},
\]
the ellipse $\mathcal E_s$ is mapped to the standard Bernstein ellipse
\[
\mathfrak{E}_{\chi_s} = \left\{\frac12\left(\xi+\xi^{-1}\right): |\xi|\leq \chi_s \right\}.
\]
The Chebyshev expansion of $\e^{h\xi}$ is
\[
\e^{h\xi} = B_0(h)+2\sum_{k=1}^{\infty} B_k(h)T_k(\xi),
\]
where $T_k$ denotes the Chebyshev polynomial of degree $k$; see, e.g.,~\cite[Chapter~5]{maha02}. We define a polynomial approximation $p_{m-1}$ for $\e^\zeta$ by truncating the Chebyshev expansion of $\e^\zeta=\e^{-h+h\xi} = \e^{-h}\e^{h\xi}$ after the $m$th term, i.e., 
\[
p_{m-1}(\zeta) := \e^{-h} \left[B_0(h) + 2\sum_{k=1}^{m-1} B_k(h) T_k\!\left(\frac{\zeta+h}{h}\right) \right].
\]
For $\xi\in \mathfrak E_{\chi_s}$,
\[
|T_k(\xi)| \leq \frac{\chi_s^k+\chi_s^{-k}}{2}.
\]
Therefore, for $\zeta\in \mathcal E_s$, we have
\[
|\e^\zeta-p_{m-1}(\zeta)| \leq 2\e^{-h} \sum_{k=m}^{\infty} B_k(h) \left|T_k\!\left(\frac{\zeta+h}{h}\right)\right| \leq \e^{-h}\sum_{k=m}^{\infty} B_k(h) \left(\chi_s^k+\chi_s^{-k}\right).
\]
Finally, since
\[
    \F(W)\subseteq \mathcal S_s\subseteq \mathcal E_s,
\]
the same estimate holds on $\mathcal F(W)$. 
Since $p_{m-1}\in \mathcal P_{m-1}$, taking the infimum over all degree-$(m-1)$ polynomials gives the stated bound for $E_{m-1}(\e^z,\F(W))$.
\end{proof}

\begin{figure}
	\centering
	\includegraphics[width=0.47\linewidth]{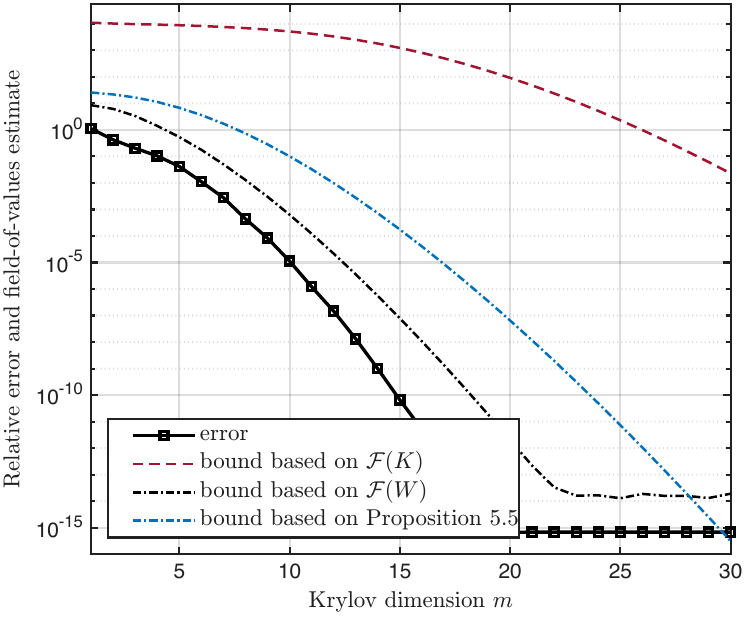}
	\caption{Comparison of bounds based on best polynomial approximation on $\F(K)$ and $\F(W)$ (which require numerically solving a polynomial approximation problem), as well as the a priori bound from~\Cref{pro:enclosing_ellipse} for the \texttt{poisson} model problem. See the description in~\Cref{subsect:five_examples} for details on the experimental setup.}
	\label{fig:poisson_proposition}
\end{figure}

In the proof of~\Cref{pro:enclosing_ellipse}, a Bunimovich stadium appears as intermediate enclosure for $\F(W)$. We refer to~\cite{varm14} for an in-depth discussion of the fact that the corresponding conformal mapping does not have a closed form. The technique of further embedding a Bunimovich stadium into an ellipse to obtain convergence estimates, similar to our derivations above, was previously used in the context of an inexact shift-invert method for Stieltjes matrix functions in~\cite[Appendix A]{gusc21}.

To illustrate the bound obtained from~\Cref{pro:enclosing_ellipse}, we perform a small numerical experiment. The setup is the same as for the \texttt{poisson} test problem considered in~\Cref{sect:experiments} below; see the description there for details. In~\Cref{fig:poisson_proposition} we plot the actual Krylov error for approximating the action of $\varphi$-functions, together with the bounds obtained from~\Cref{pro:enclosing_ellipse} as well as the bounds obtained via optimal polynomial approximation on $\F(K)$ and $\F(W)$ (computed as described in~\Cref{appendix:poly_error}). First, we observe that the bound based on polynomial approximation on $\F(W)$ very accurately predicts the actual convergence behavior of the Krylov method, while the bound obtained for $\F(K)$ is not predictive at all, having a wrong slope and overestimating the magnitude of the error by several orders of magnitude. The bound from~\Cref{pro:enclosing_ellipse} lies in between these two, as expected. Being based on polynomial approximation on a set that is larger than $\F(W)$, it overestimates the magnitude of the error a bit more and also does not predict the convergence slope completely correctly, in particular in the second phase of faster convergence. Still, it represents a substantial improvement over the bound based on $\F(K)$, and, in contrast to the other two bounds, can be cheaply evaluated without the need to numerically solve complicated polynomial approximation problems.

\section{Numerical experiments}\label{sect:experiments}

\begin{figure}[t]
	\centering
	\begin{subfigure}{0.47\linewidth}
		\centering
		\includegraphics[width=\linewidth]{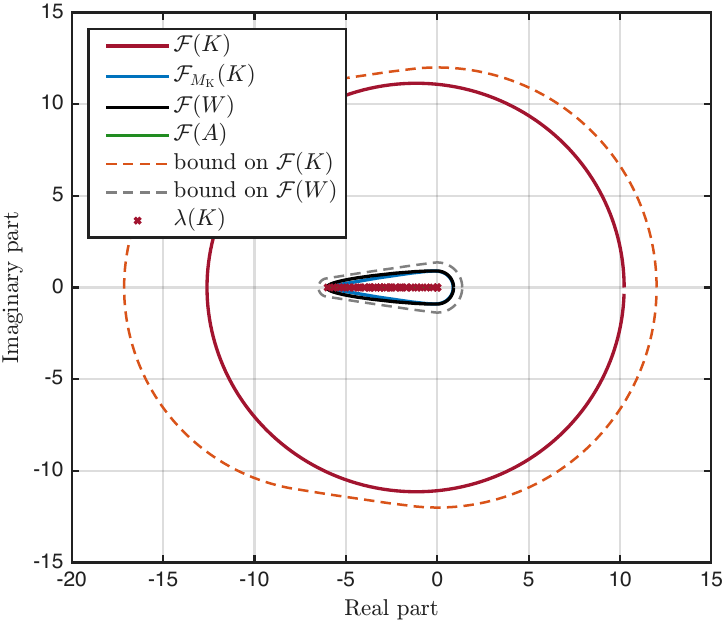}
	\end{subfigure}
	\hfill
	\begin{subfigure}{0.47\linewidth}
		\centering
		\includegraphics[width=\linewidth]{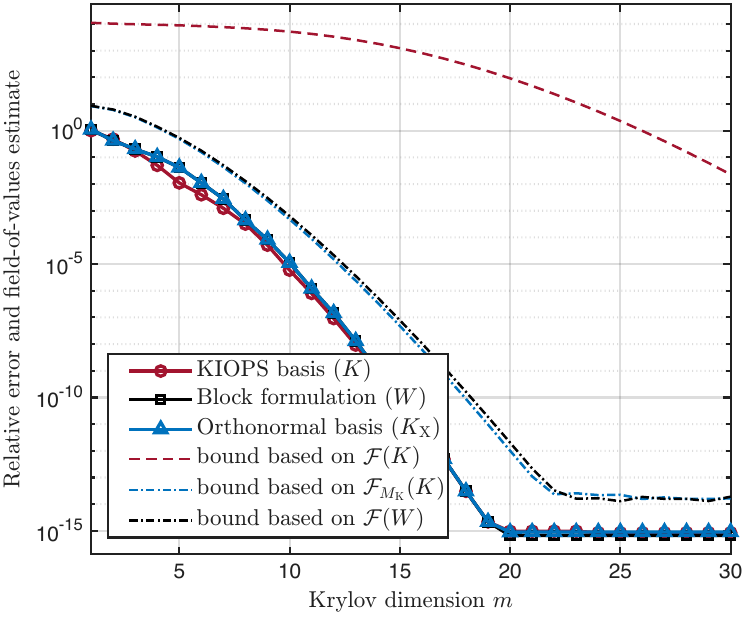}
	\end{subfigure}
	\caption{Negative Poisson matrix from a two-dimensional finite-difference Laplacian.}
	\label{fig:test-poisson}
\end{figure}

In this section we demonstrate the effectiveness of the new error bounds and results derived in the previous sections.
We compare error bounds obtained from different field-of-values sets using characteristic numerical examples of modest size. 
This is necessary for the sampled boundary computations and the discrete minimax problems in~\Cref{appendix:poly_error} to remain feasible within a reasonable amount of time. 

All experiments were run using MATLAB R2025b Update 5 on a Mac laptop equipped with a 2.0 GHz quad-core Intel Core i5-1038NG7 processor and 16 GiB of RAM. The code we used to produce the results in this section is available on GitHub\footnote{\url{https://github.com/xiaobo-liu/phimv_krylov_aug}.}.

For all experiments below, the input matrix $A$ is scaled as $A \gets \alpha A/\normt{A}$ with $\alpha=6$. This moderate scaling is not special, but makes the sizes of the fields of values comparable across different experiments (since the field of values scales linearly).
The Poisson example has dimension $49$ while the rest have dimension $40$. 
All test matrices are signed or shifted so that their spectra lie in the left half-plane, reflecting the dissipative linear operators arising in applications of exponential integrators~\cite{hoos10}.

We take $s=5$ in~\eqref{eq:phi-lin-comb}, a modest augmented KIOPS dimension such that the auxiliary Jordan block is large enough to produce a visible disk $\F(J_s(0))$ in~\eqref{eq:fov-J}.
The entries of the vector $b_0$ are drawn from a Gaussian distribution and normalized to unit $2$-norm. The remaining vectors are chosen from the family
\begin{equation}\label{eq:bj}
	b_j =
	\beta \; \frac{q+\delta r_j}
	{\normt{q+\delta r_j}},
	\quad j=1,\ldots,s,
\end{equation}
where $q$ and the $r_j$ are unit vectors obtained by normalizing independent standard Gaussian vectors. The common factor $\beta$ scales the upper-right block of the KIOPS matrix $K$ in
\eqref{eq:aug-mat-kiops} linearly, and the collinearity of the $b_j$ is affected by the perturbation level $\delta$. 
A larger $\beta$ may enlarge $\F(K)$ and make the estimate in~\eqref{eq:fov-K} more pessimistic, because the Euclidean field of values is sensitive to the scaled off-diagonal block. In contrast, it leaves $\F(W)$ and the corresponding scaled metric field of values unchanged, since the scaling is compensated by the metric $\overline \MK(\beta)$ from~\Cref{cor:rhs-scaling-metric-distortion}; see~\eqref{eq:fov-MK-K-beta}.
For the five model examples considered in \Cref{subsect:five_examples}, we take $\beta=10$ and $\delta = 10^{-1}$ in~\eqref{eq:bj}. The right-hand-side sensitivity tests in \Cref{subsect:rhs_sens} then vary either $\delta$ or $\beta$ while keeping the other parameter fixed at this reference value.

\subsection{Model examples}\label{subsect:five_examples}

\begin{figure}[t]
	\centering
	\begin{subfigure}{0.47\linewidth}
		\includegraphics[width=\linewidth]{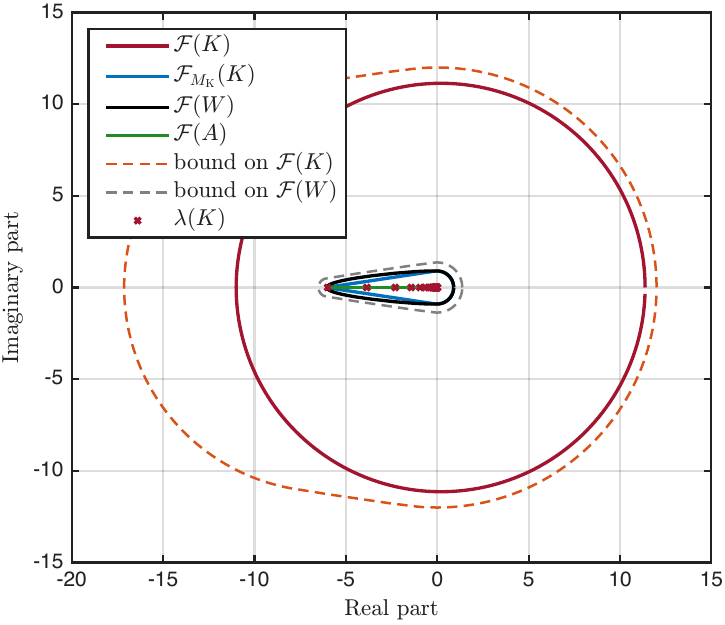}
	\end{subfigure}
	\hfill
	\begin{subfigure}{0.47\linewidth}
		\includegraphics[width=\linewidth]{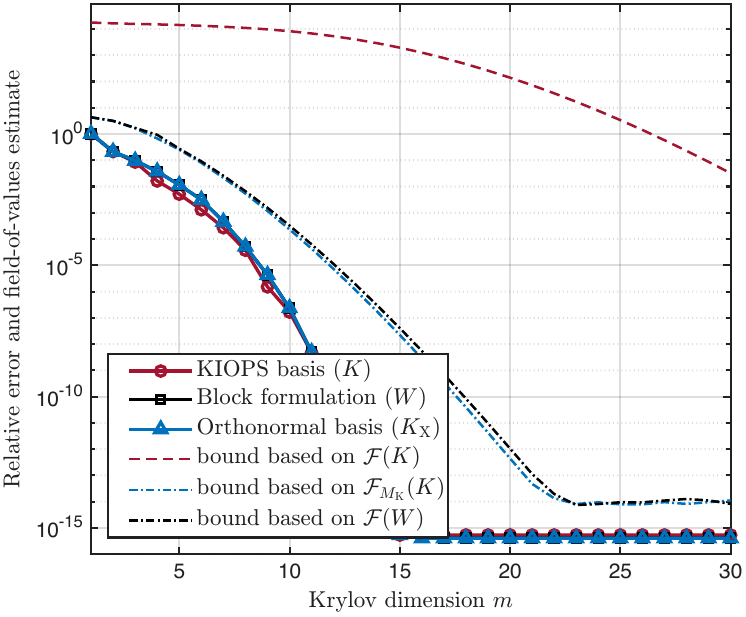}
	\end{subfigure}
	\caption{Hermitian negative definite Kac--Murdock--Szeg\"{o} Toeplitz matrix.}
	\label{fig:test-kms}
\end{figure}

We first compare three Arnoldi formulations for computing the same projected exponential action associated with~\eqref{eq:phi-lin-comb}:
\begin{itemize}
	\item KIOPS basis: Arnoldi is applied to the KIOPS matrix $K$ in
	\eqref{eq:aug-mat-kiops}.
	\item Full block formulation: Arnoldi is applied to the block matrix $W$ in
	\eqref{eq:aug-mat-blk}.
	\item Orthonormal basis: Arnoldi is applied to the matrix $\KX$ in
	\eqref{eq:KX-general} with $X=\XK C_{\rm K}^{-1}$, where
	$C_{\rm K}=(\XK^*\XK)^{1/2}$ and $\XK$ is defined in~\eqref{eq:XK}.
\end{itemize}
In all three cases, the Arnoldi iteration uses modified Gram--Schmidt orthogonalization with one reorthogonalization pass. 
To isolate the effects of the Arnoldi formulation and the associated field-of-values bounds, the adaptive substepping and incomplete orthogonalization of the practical KIOPS algorithm~\cite{grt18} are not implemented.
Note that the full block and orthonormal-basis formulations differ only by an isometric coordinate representation of the invariant subspace in~\Cref{thm:krylov_W_KX}.

We discuss five model experiments, each displayed in a two-panel figure. The left panel compares the sampled boundaries of $\F(K)$, $\F(W)$, and $\F_{\MK}(K)$, along with $\F(A)$ (filled in light color), which enables us to see how the augmentation enlarges the geometry of the spectral region of the original matrix. It also shows the enclosing regions from~\eqref{eq:fov-W} and~\eqref{eq:fov-K}, together with the eigenvalues of $K$. 
The right panel reports the relative error $\normt{y_m-y}/\normt{y}$ for the three Arnoldi realizations for $m=1,\ldots,30$, where the reference solution vector $y$ of~\eqref{eq:phi-lin-comb} is computed by~\eqref{eq:poly-krylov-KIOPS} with explicit assembly of the augmented matrix $K$ and the matrix exponential computed by MATLAB's \texttt{expm}~\cite{alhi09}. These errors are compared with the corresponding bounds computed from the sampled sets $\F(K)$, $\F_{\MK}(K)$, and $\F(W)$ by the discrete best polynomial approximation (computed as described in~\Cref{appendix:poly_error}).

\paragraph{Example 1}
The first model problem uses $\texttt{-gallery('poisson',7)}$, the five-point discretization of the negative two-dimensional Laplacian on a $7\times7$ grid. 
It corresponds to the discretized dissipative diffusion operator arising in heat and reaction--diffusion equations~\cite{hls98} and is a standard test problem for actions of the matrix exponential and $\varphi$-functions in exponential integrators~\cite{alhi11}.
Here, $A$ is negative definite and $\F(A)$ is an interval,  which is precisely when~\Cref{pro:enclosing_ellipse} is applicable; see~\Cref{fig:poisson_proposition} for the error bound obtained from it.

The results are presented in~\Cref{fig:test-poisson}.
In the left panel, $\F(W)$ and the metric set $\F_{\MK}(K)$ are small enlargements of the interval $\F(A)$ on the negative real axis, while $\F(K)$ expands into a much larger region. We clearly observe that $\F_{\MK}(K)\subseteq \F(W)$ (with $\F_{\MK}(K)$ nearly filling $\F(W)$, in fact), as proved in~\Cref{lem:metric_fov_kx}.
In the right panel, the behavior of the three Arnoldi realizations is almost indistinguishable, reaching unit roundoff accuracy after about $20$ steps.
The bounds based on $\F(W)$ and $\F_{\MK}(K)$ follow the observed error decay throughout, whereas the bounds based on $\F(K)$ are always at least several orders of magnitude too pessimistic.

\paragraph{Example 2}
We use the Kac--Murdock--Szeg\"{o} Toeplitz matrix~\cite{kms53} with entries $a_{ij} = -(0.88^{|i-j|})$, generated by 
$\texttt{-gallery('kms',40,0.88)}$.
The value $0.88$ of the correlation parameter is not special; it makes the scaled spectrum stretch from $-6$ to a point close to the origin with strong spectral clustering, in contrast with the more evenly spread spectrum of the Poisson example above.

Here $\F(A)$ is also an interval, and the field-of-values plot on the left of~\Cref{fig:test-kms} exhibits very similar geometry as in \emph{Example~1}.
The sets $\F(W)$ and $\F_{\MK}(K)$ are only slightly larger than $\F(A)$, unlike $\F(K)$, and the block-based enclosure in~\eqref{eq:fov-W} again provides fairly good bounds for $\F(W)$ and $\F_{\MK}(K)$. While the bound~\eqref{eq:fov-K} is a less tight enclosure than~\eqref{eq:fov-W}, it still predicts the relevant region of the complex plane reasonably well.
In the convergence plot of~\Cref{fig:test-kms}, the error decay phase is slightly shorter than in the Poisson case, with about $15$ steps. The behavior is very accurately predicted by the metric and the larger block bounds in the initial phase of slower convergence, but the bounds become less tight once convergence speeds up after a few iterations.
The bound based on $\F(K)$, on the contrary, remains not informative at all.

\begin{figure}[t]
	\centering
	\begin{subfigure}{0.47\linewidth}
		\centering
		\includegraphics[width=\linewidth]{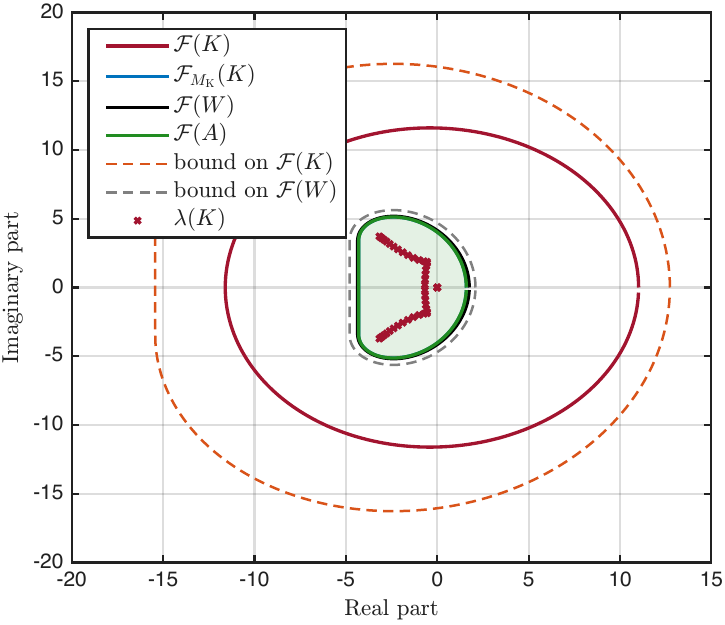}
	\end{subfigure}
	\hfill
	\begin{subfigure}{0.47\linewidth}
		\centering
		\includegraphics[width=\linewidth]{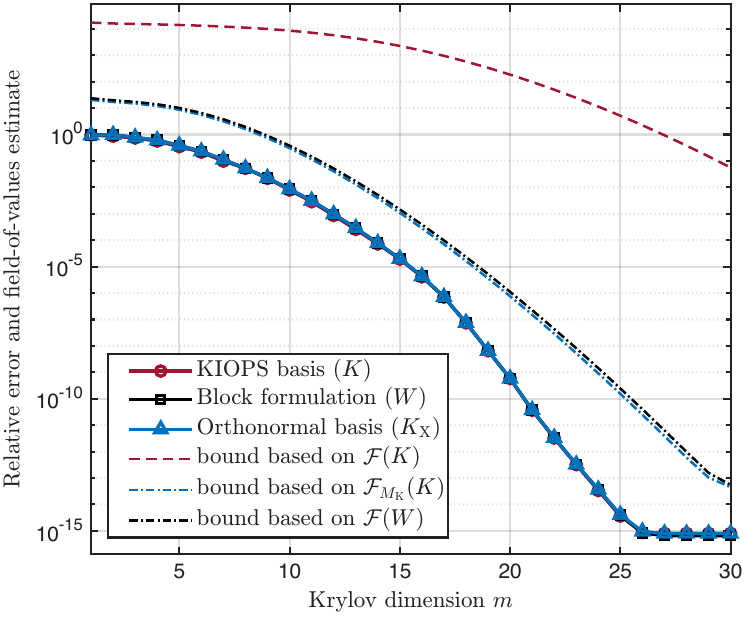}
	\end{subfigure}
	\caption{Shifted Grcar matrix, a nonnormal Toeplitz test problem.}
	\label{fig:test-grcar}
\end{figure}

\paragraph{Example 3}
Our third example uses the Grcar matrix~\cite{grcar89}, a standard nonnormal Toeplitz example in the study of pseudospectra and behavior of nonnormal matrices~\cite{benz20, tref92, trem05}, especially for demonstrating that eigenvalue information alone can be misleading for the convergence of Krylov methods.
The test matrix we use is generated by $\texttt{gallery('grcar',40)-2*eye(40)}$, where the shift by $-2I$ keeps the eigenvalues in the left half-plane after scaling, yet the field of values still protrudes substantially farther to the right than the spectrum. 

As we can see from~\Cref{fig:test-grcar}, in the left panel, $\F(A)$ is already two-dimensional and visibly nonnormal, but $\F(W)$ and $\F_{\MK}(K)$ still remain close to each other, whereas $\F(K)$ is much larger. 
In the right panel, the error decreases more gradually than in the previous Hermitian examples, and the metric and the larger block bounds reproduce this slower slope of convergence far better than the bound based on $\F(K)$.

\begin{figure}[t]
	\centering
	\begin{subfigure}{0.47\linewidth}
		\centering
		\includegraphics[width=\linewidth]{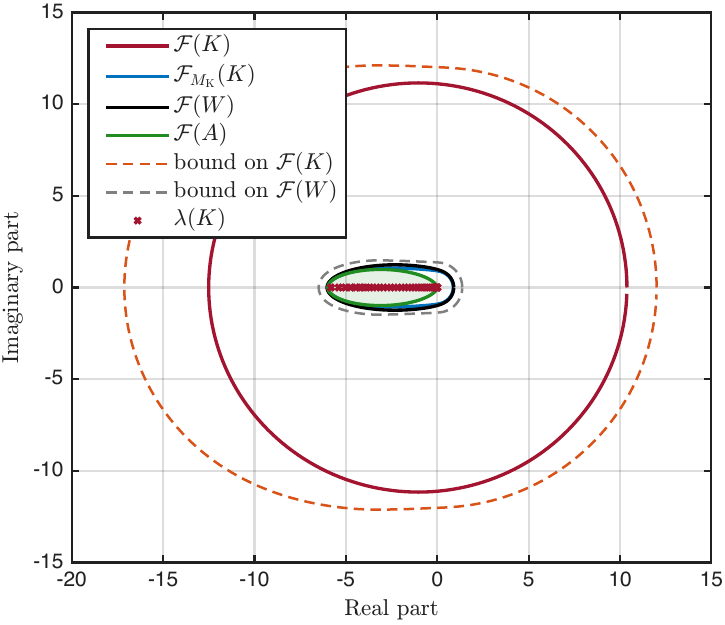}
	\end{subfigure}
	\hfill
	\begin{subfigure}{0.47\linewidth}
		\centering
		\includegraphics[width=\linewidth]{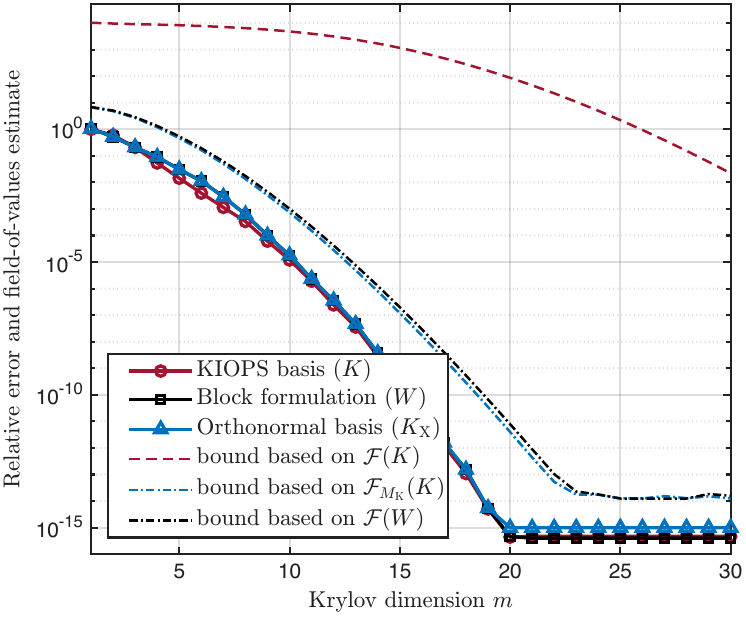}
	\end{subfigure}
	\caption{Negative Dorr matrix, a nonsymmetric ill-conditioned tridiagonal test problem.}
	\label{fig:test-dorr}
\end{figure}

\paragraph{Example 4}
The fourth example uses the negative of the Dorr matrix~\cite{dorr71}, generated by $\texttt{-gallery('dorr',n)}$. This matrix is tridiagonal, row diagonally dominant, and ill-conditioned, presenting a finite-difference-type but nonsymmetric matrix from a singular perturbation problem.
Although the eigenvalues of $A$ lie on the negative real axis, the nonsymmetry makes $\F(A)$ two-dimensional and allows part of $\F(A)$ to cross the imaginary axis. 
We are therefore testing whether the error bounds respond to the rightmost extent of the field of values, a feature absent from the Hermitian examples, where $\F(A)$ is just the convex hull of the eigenvalues.

In~\Cref{fig:test-dorr}, despite this non-Hermitian field-of-values geometry, the field-of-values and Krylov convergence plots show trends rather similar to those observed in the Poisson example.

\begin{figure}[t]
	\centering
	\begin{subfigure}{0.47\linewidth}
		\centering
		\includegraphics[width=\linewidth]{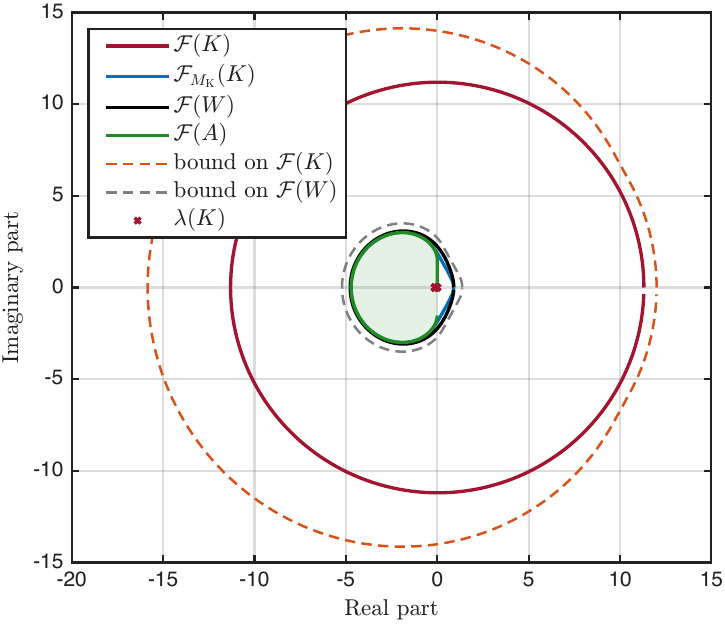}
	\end{subfigure}
	\hfill
	\begin{subfigure}{0.47\linewidth}
		\centering
		\includegraphics[width=\linewidth]{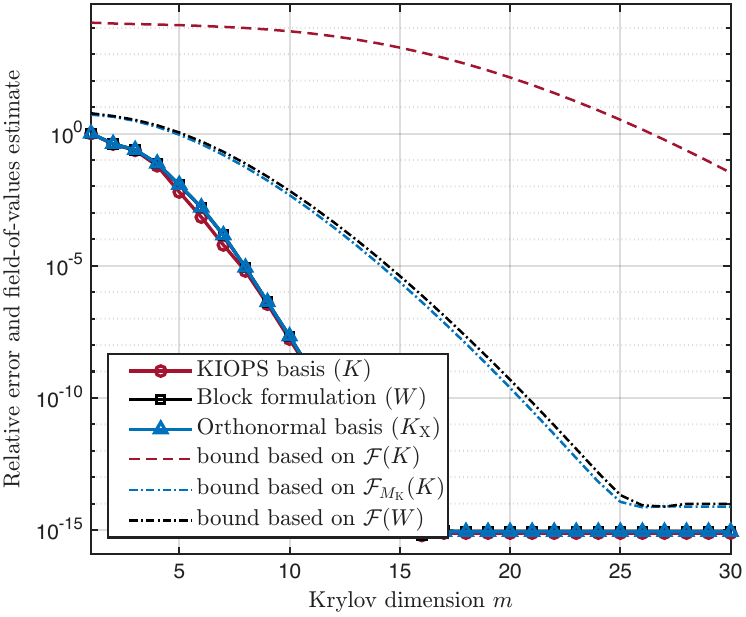}
	\end{subfigure}
	\caption{Shifted triangular Toeplitz matrix with strongly nonnormal behavior.}
	\label{fig:test-triw}
\end{figure}

\paragraph{Example 5}
The final experiment concerns a shifted version of the triangular Toeplitz matrix discussed by Kahan~\cite{kaha66} and by Golub and Wilkinson~\cite{gowi76}, generated via the command $\texttt{gallery('triw',40)-1.5*eye(40)}$.
After scaling, all eigenvalues of $A$ coincide at near $-0.12$ on the negative real axis, while the field-of-values extends to a much larger region and comes close to the imaginary axis. 
This presents a difficult nonnormal case in which the spectrum gives little information about the geometry relevant to field-of-values and pseudospectral convergence estimates~\cite{benz20, tref92, waye17}. 

The results in~\Cref{fig:test-triw} show that, even for this highly nonnormal example, $\F(W)$ and $\F_{\MK}(K)$ are again only slight enlargements of $\F(A)$ and are well approximated by the enclosure in~\eqref{eq:fov-W}, whereas $\F(K)$ is very pessimistic.
On the other hand, the error of the Krylov methods drops rapidly to unit round-off level, while the bounds based on $\F(W)$ and $\F_{\MK}(K)$ are too conservative, not able to accurately capture the slope of convergence.
The $\F(K)$-based bound is again dominated by the enlargement caused by using the ordinary Euclidean inner product on the nonorthogonal KIOPS coordinate basis.

\subsection{Sensitivity to right-hand-side scaling and collinearity}\label{subsect:rhs_sens}

The second set of experiments tests the dependence of the field of values and convergence bounds on the scaling and collinearity of the right-hand side vectors.
We use the Poisson and Grcar matrices previously used in the experiments of~\Cref{subsect:five_examples}, representing the Hermitian and nonnormal cases, respectively.
For each matrix, a single random realization of $q$ and the vectors $r_j$ of~\eqref{eq:bj} is generated and then reused throughout the sensitivity tests, so that the observed changes are due only to the prescribed values of $\beta$ and $\delta$.
In the sensitivity test on $\delta$, $\beta=10$ is fixed and $\delta$ is chosen as $\{0, 10^{-3}, 10^{-1}, 10\}$, while in the sensitivity test on $\beta$, we set $\delta=10^{-1}$ throughout and $\beta$ takes values from $\{1, 4, 10, 25\}$. 

\begin{figure}[t]
	\centering
	\begin{subfigure}{0.48\linewidth}
		\centering
		\includegraphics[width=\linewidth]{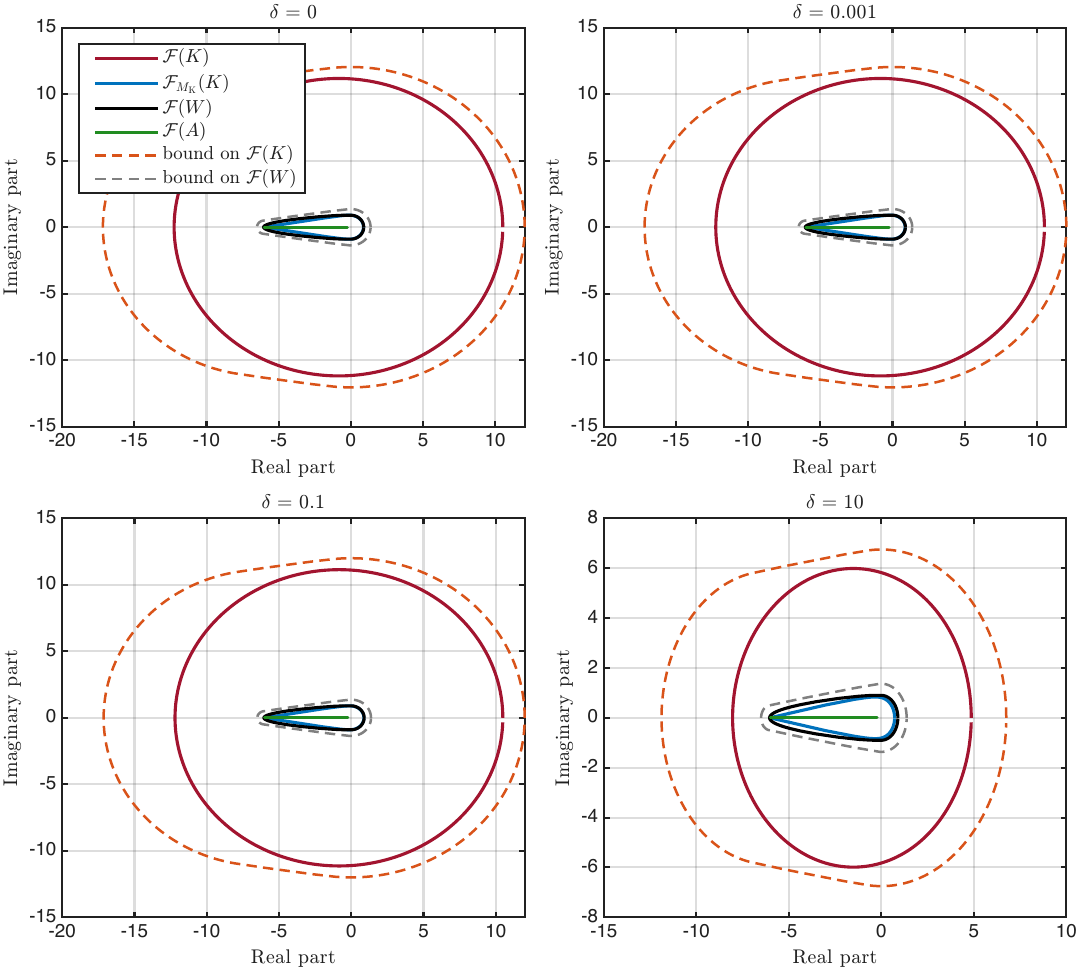}
	\end{subfigure}
	\hfill
	\begin{subfigure}{0.48\linewidth}
		\centering
		\includegraphics[width=\linewidth]{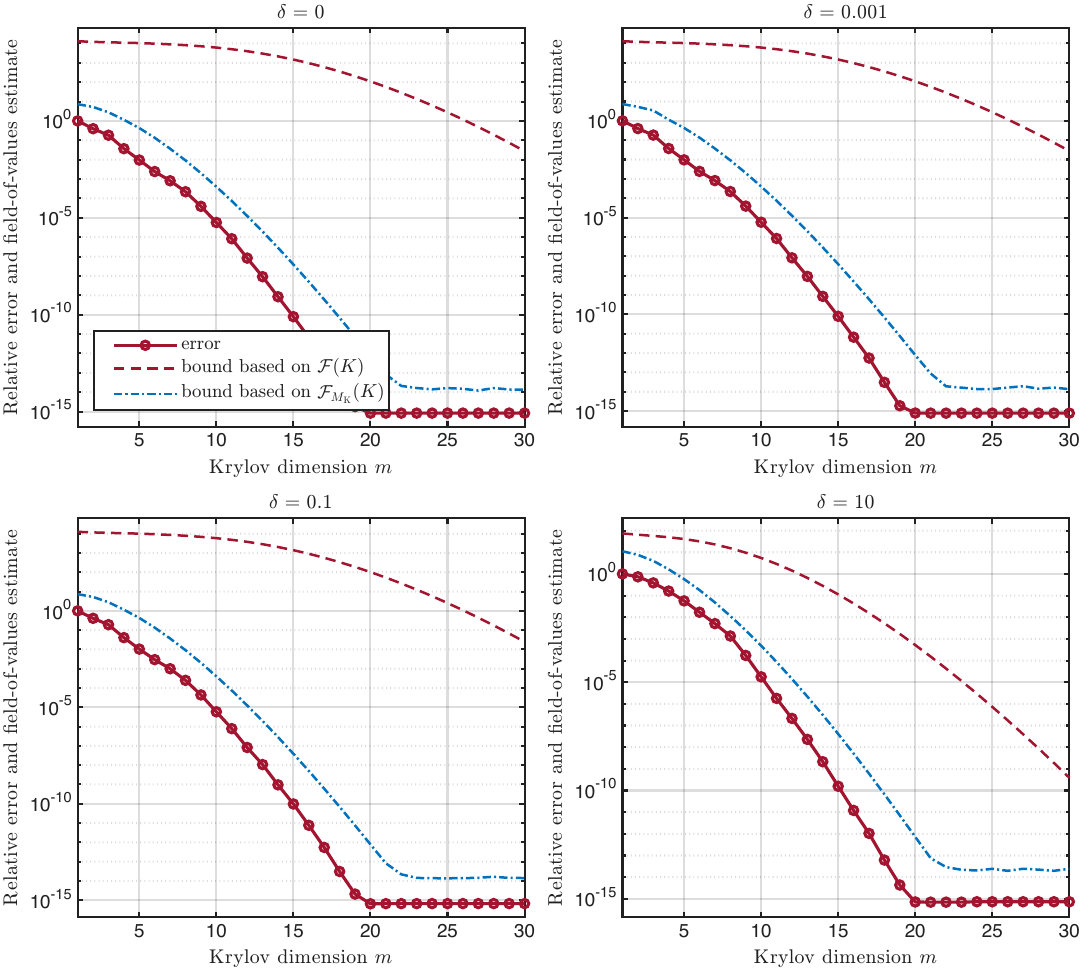}
	\end{subfigure}
	\caption{Right-hand-side sensitivity test for the Poisson matrix with fixed $\beta=10$ and varying $\delta$.}
	\label{fig:rhs-sens-poisson-delta}
\end{figure}

\begin{figure}
	\centering
	\begin{subfigure}{0.48\linewidth}
		\centering
		\includegraphics[width=\linewidth]{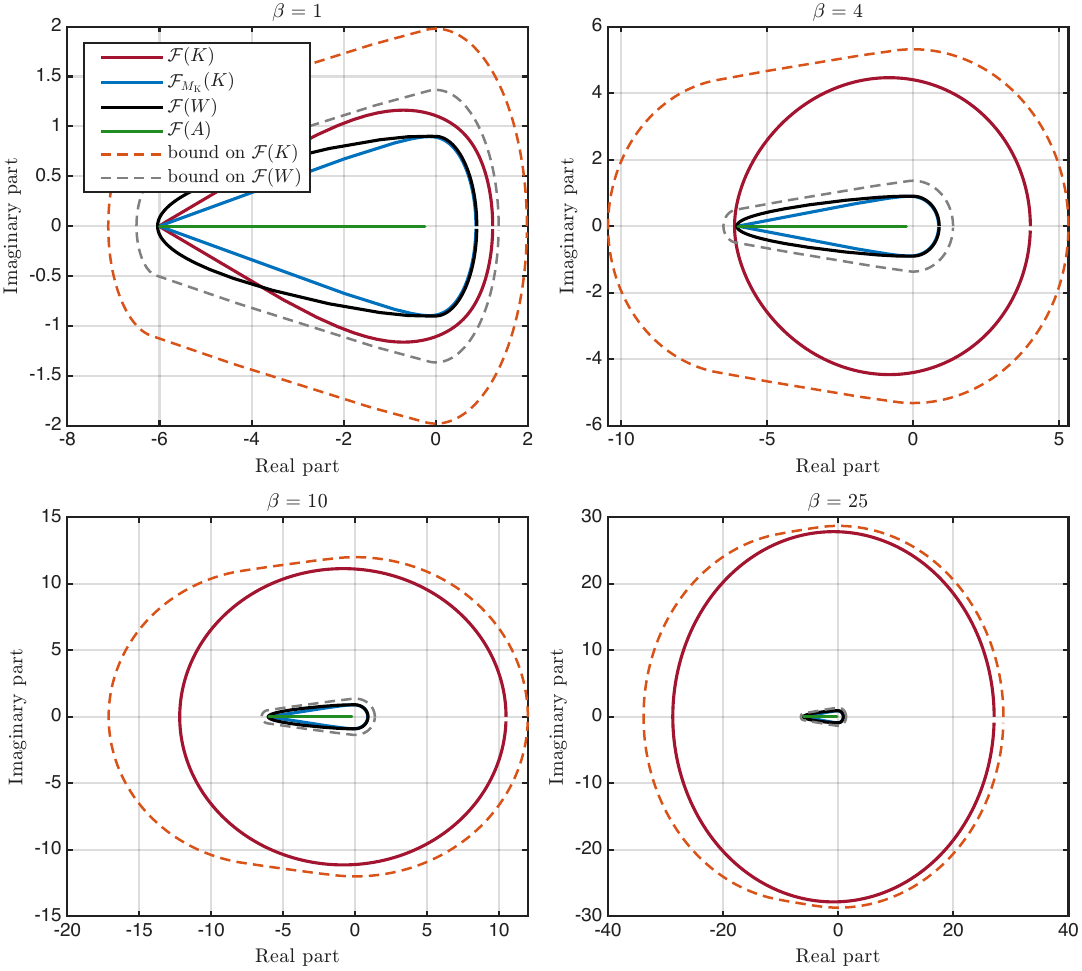}
	\end{subfigure}
	\hfill
	\begin{subfigure}{0.48\linewidth}
		\centering
		\includegraphics[width=\linewidth]{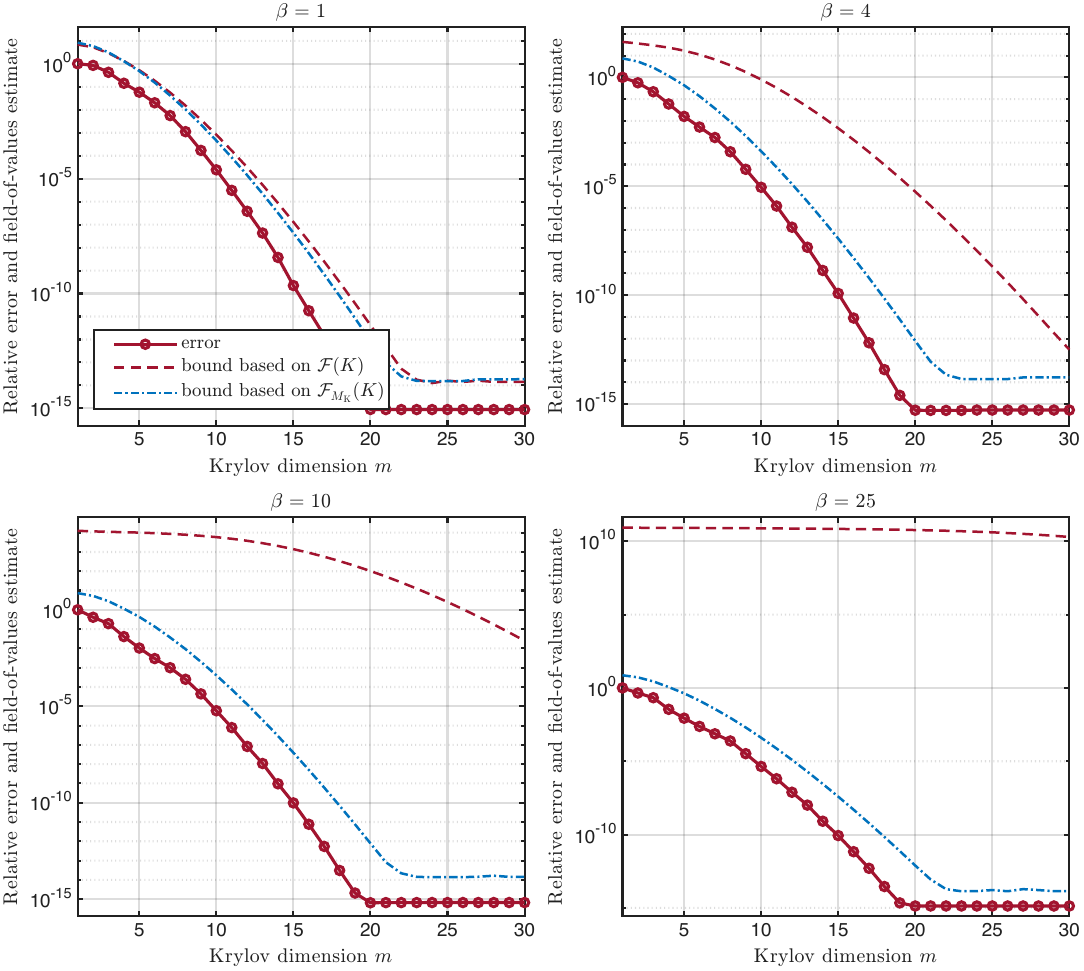}
	\end{subfigure}
	\caption{Right-hand-side sensitivity test for the Poisson matrix with fixed $\delta=10^{-1}$ and varying $\beta$.}
	\label{fig:rhs-sens-poisson-beta}
\end{figure}

For the Poisson matrix,~\Cref{fig:rhs-sens-poisson-delta,fig:rhs-sens-poisson-beta} show the effect of the right-hand side parameters in a setting where $\F(A)$ is an interval.
When $\delta$ is varied with $\beta=10$ fixed, the vectors $b_j$ move from an exactly collinear configuration at $\delta=0$ to increasingly separated directions, and we have $\normt{B}\leq \normF{B}=\sqrt{s}\beta \approx 22.36$.
We found that $\sigma_{\min}(\XK)\approx 5.21$ and $\sigma_{\max}(\XK)\approx 35.13$ for $\delta=0$ and $\delta=10^{-3}$, but this singular value gap shrinks to approximately $[9.91,22.41]$ for $\delta=10$.
This explains the slight contraction of the Euclidean set $\F(K)$ as $\delta$ increases, shown on the left panel of~\Cref{fig:rhs-sens-poisson-delta}. By~\Cref{thm:basis-metric-distortion}, the reduction in the extremal singular values makes the metric field-of-values less anisotropic and more aligned to the standard Euclidean metric.
In all cases, the metric set $\F_{\MK}(K)$ remains close to the block field of values $\F(W)$.
This varying field-of-values behavior is reflected in the convergence plot, where the error bound based on $\F(K)$ is less pessimistic as $\delta$ increases, although it remains a few orders of magnitude above the observed error as the Krylov dimension grows.

When $\beta$ is varied with $\delta = 0.1$ fixed,~\Cref{fig:rhs-sens-poisson-beta} showcases a much more dramatic effect. In this case, increasing $\beta$ from $1$ to $25$ changes $\normt{B}$ linearly from about $2.23$ to about $55.68$ and scales $[\sigma_{\min}(\XK),\sigma_{\max}(\XK)]$ from approximately $[0.54,3.50]$ to $[13.59,87.54]$.
The Euclidean field of values $\F(K)$ and its enclosure in~\eqref{eq:fov-K} become increasingly pessimistic as the right-hand-side scaling factor $\beta$ increases. Correspondingly, the convergence bound based on $\F(K)$ rapidly fails to provide any useful prediction of the true error.
In contrast, the metric set $\F_{\MK}(\beta)$ shown in the plots is insensitive to the scaling; this is the plotted version of the scaled-metric invariance in~\eqref{eq:fov-MK-K-beta}, with $K$ and $\MK$ understood for the current scaled data.

\begin{figure}[t]
	\centering
	\begin{subfigure}{0.48\linewidth}
		\centering
		\includegraphics[width=\linewidth]{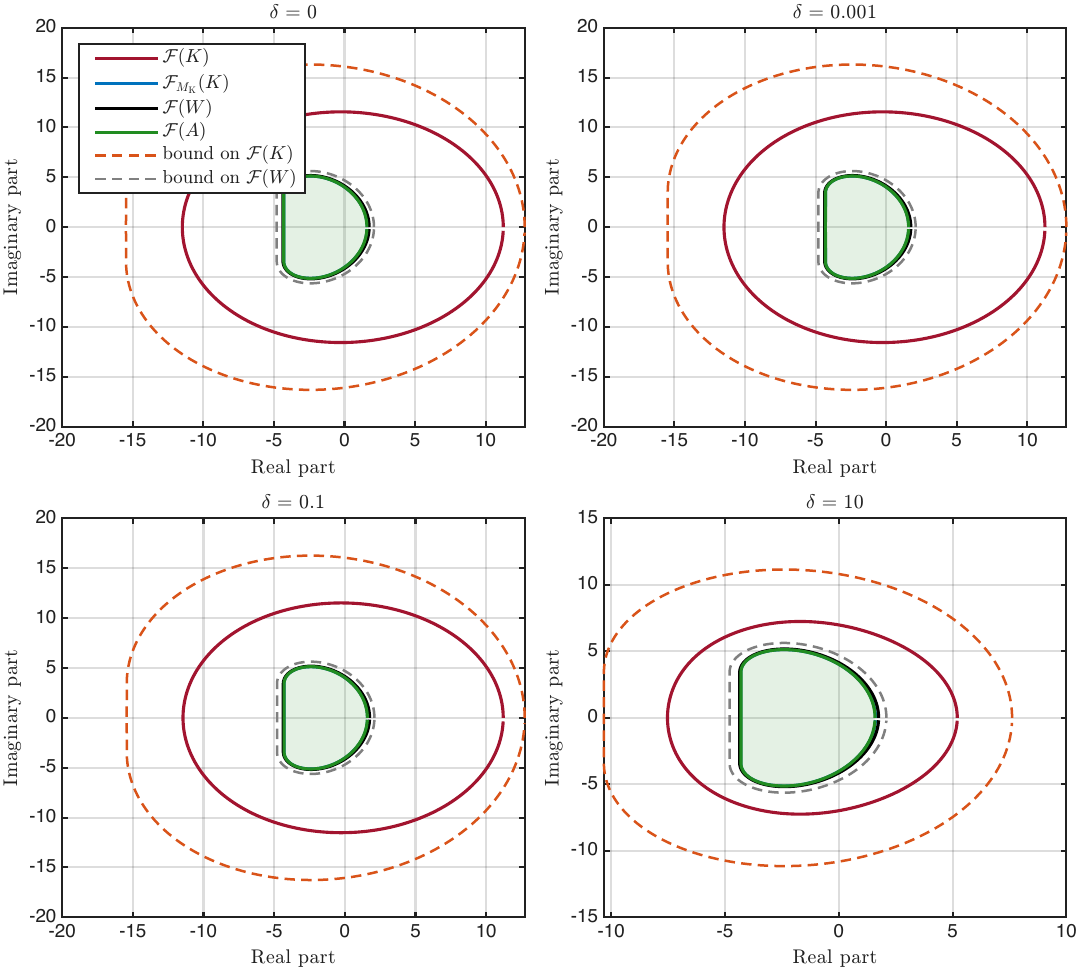}
	\end{subfigure}
	\hfill
	\begin{subfigure}{0.48\linewidth}
		\centering
		\includegraphics[width=\linewidth]{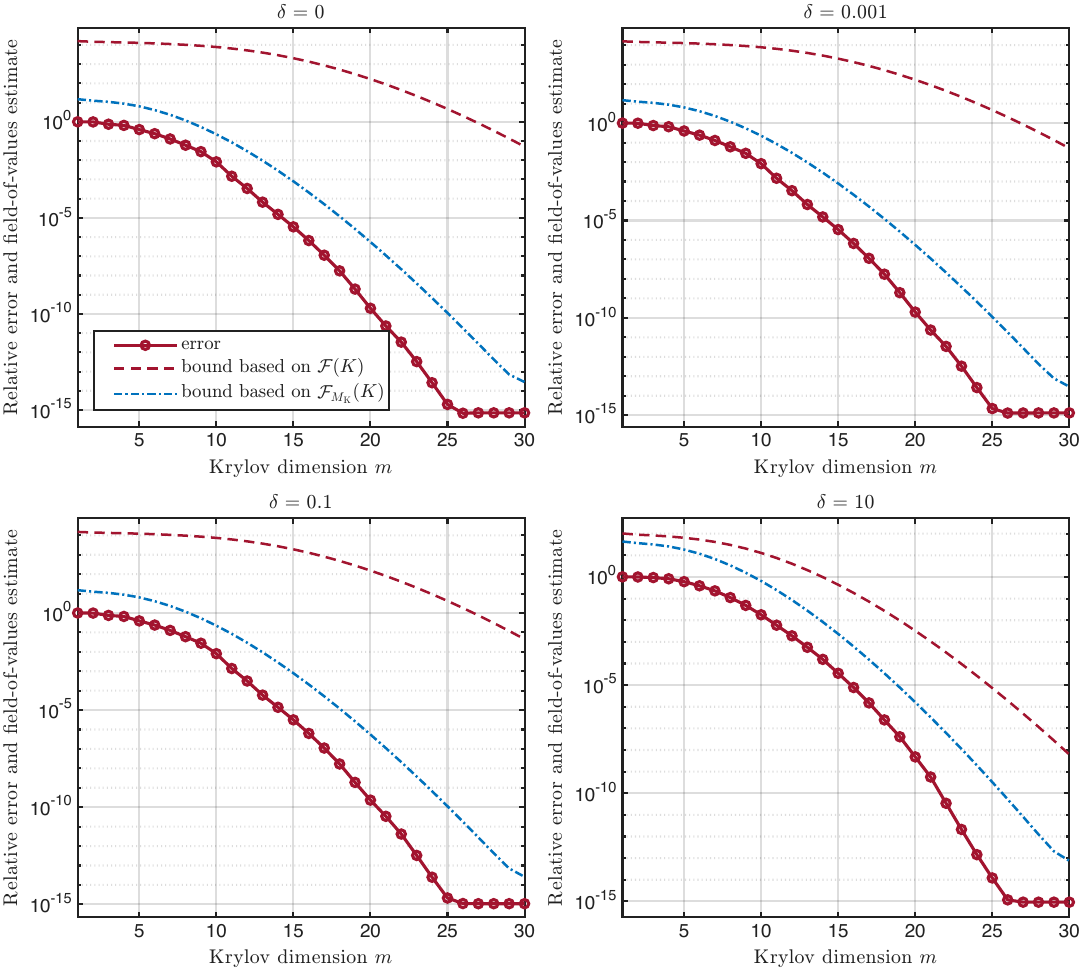}
	\end{subfigure}
	\caption{Right-hand-side sensitivity test for the Grcar matrix with fixed $\beta=10$ and varying $\delta$.}
	\label{fig:rhs-sens-grcar-delta}
\end{figure}

\begin{figure}[t]
	\centering
	\begin{subfigure}{0.48\linewidth}
		\centering
		\includegraphics[width=\linewidth]{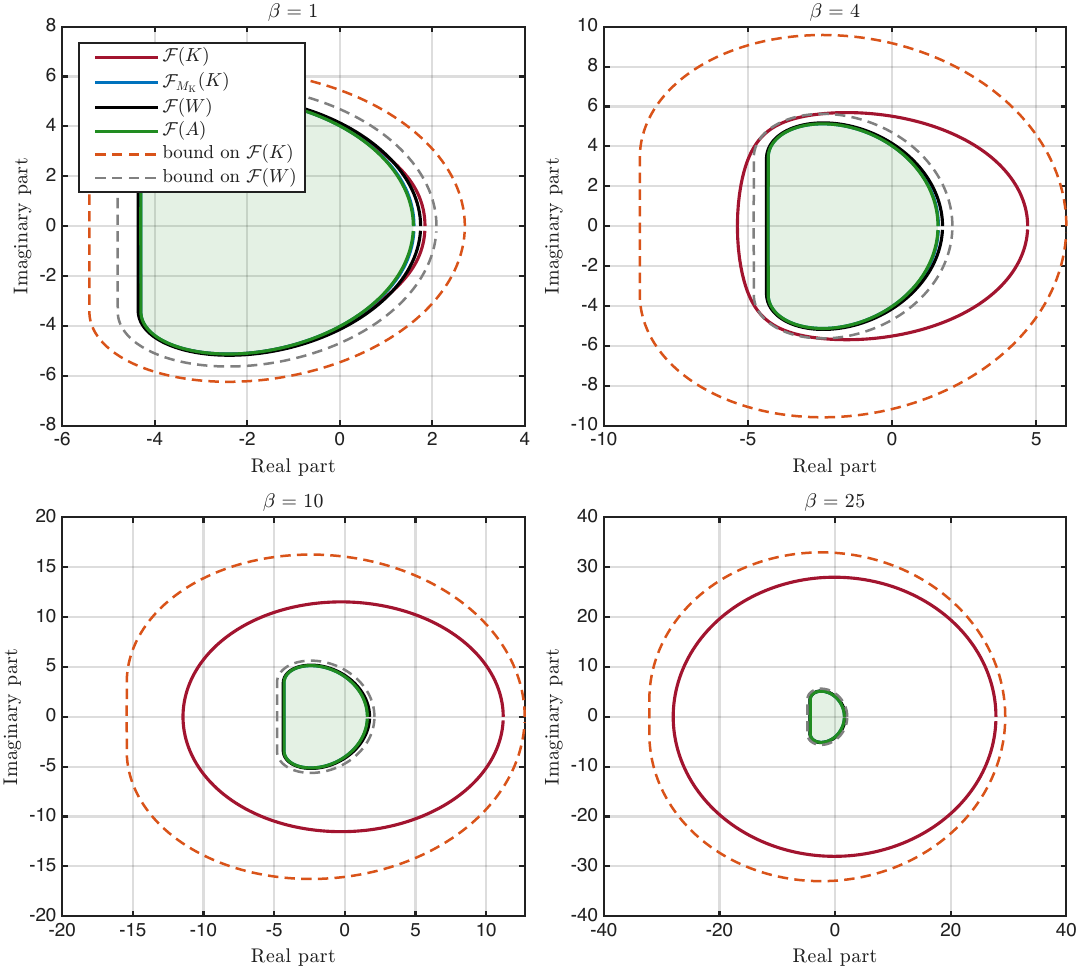}
	\end{subfigure}
	\hfill
	\begin{subfigure}{0.48\linewidth}
		\centering
		\includegraphics[width=\linewidth]{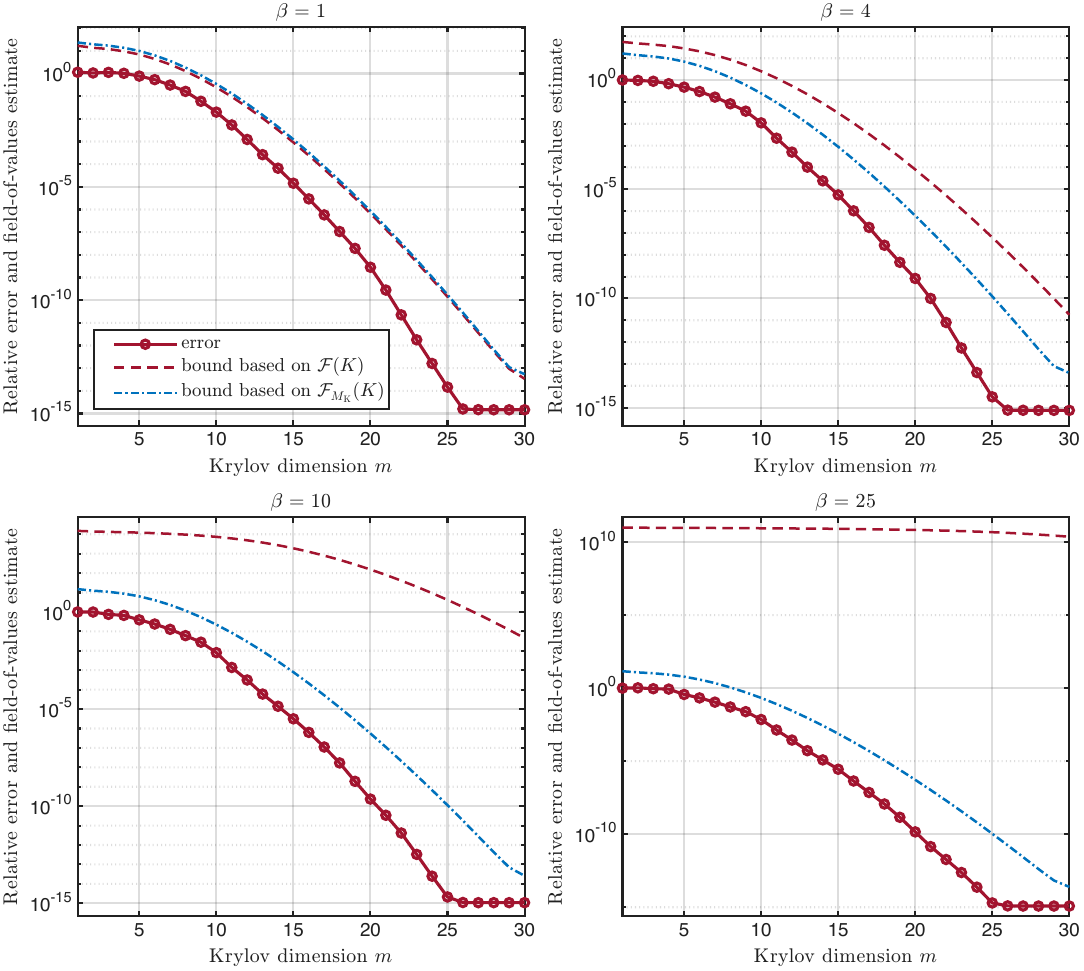}
	\end{subfigure}
	\caption{Right-hand-side sensitivity test for the Grcar matrix with fixed $\delta=10^{-1}$ and varying $\beta$.}
	\label{fig:rhs-sens-grcar-beta}
\end{figure}

The sensitivity tests on the Grcar matrix, presented in~\Cref{fig:rhs-sens-grcar-delta,fig:rhs-sens-grcar-beta}, confirm the same mechanism in the presence of nonnormal $A$, and a two-dimensional $\F(A)$.
The metric artifact in the Euclidean coordinate representation is exposed mainly by right-hand-side scaling $\beta$, and to a lesser extent by the collinearity of the right-hand-side vectors, controlled by $\delta$ in the experiments. 
As this scaling increases, $\F(K)$ and the associated convergence bound deteriorate, whereas the associated metric field of values $\F_{\MK}(K)$ stays tied to that of the block formulation~\eqref{eq:aug-mat-blk} and yields improved Krylov error bounds.

A somewhat surprising observation from \Cref{fig:rhs-sens-grcar-beta} is that for $\beta = 1$, the bound based on $\mathcal{F}(K)$ lies slightly \emph{below} the bound based on the metric field of values, despite the latter set being a subset of the former. This is due to the fact that the prefactor in the bound based on $\F(K)$ involves $\normt{c} = \sqrt{\normt{b_0}^2+1}$, while the bound based on the metric field of values involves the possibly larger constant $\normt{b} = \sqrt{\normt{b_0}^2+\normt{\bb}^2}$. In a situation like here, where the metric field of values is only very slightly smaller than $\F(K)$, this prefactor can thus make a difference in determining which of the bounds is tighter.

\section{Conclusions}\label{sect:conclusions}
We have revisited Krylov subspace methods for computing linear combinations of matrix $\varphi$-function actions via the exponential of an augmented matrix, focusing on the established KIOPS algorithm~\cite{grt18} in light of the recent block triangular formulation developed in~\cite{alli26}.
We showed that the KIOPS matrix $K$ in~\eqref{eq:aug-mat-kiops} is the coordinate representation of the block matrix $W$ of~\eqref{eq:aug-mat-blk} restricted to the invariant subspace $\C^n\oplus\mathcal K_s(J,\bb)$ and identified an embedding relation between the Krylov spaces built with $K$ and $W$; see~\Cref{thm:krylov_W_KX} and~\Cref{cor:exponential_error}. 
Thus, for every polynomial $p$, applying $p(K)$ and $p(W)$ to the respective starting vectors gives the same result after projection onto the first $n$ components.

Viewing the basis produced by an augmentation scheme such as KIOPS as a coordinate representation of the larger block formulation motivates the use of a different inner product, and hence a different metric for the field of values.
We showed that the difference between this \emph{induced metric} and the (ordinary) Euclidean metric is governed by the singular values of the basis matrix; see~\Cref{thm:basis-metric-distortion}.
For the KIOPS basis, this implies that scaling the input vectors $b_j$ can substantially change the metric distortion and hence affects the associated Krylov error bound based on the Euclidean metric; see~\Cref{cor:rhs-scaling-metric-distortion}. In comparison, the error bounds based on the corresponding scaled metric field of values in~\eqref{eq:fov-MK-K-beta} are invariant under this scaling.
We then proved that, for a general basis matrix $X$, the metric field of values $\F_{\MX}(\KX)$ is contained in $\F(W)$, the Euclidean field of values of $W$; see~\Cref{lem:metric_fov_kx}. 
This leads to much better convergence bounds than those based on the Euclidean field of values $\F(K)$, which depends on the input vectors $b_j$.
When the input matrix $A$ is negative semidefinite, we derived an easily computable bound for best polynomial approximation on $\F(W)$ via an enclosing Bernstein ellipse and the associated truncated Chebyshev series estimate; see~\Cref{pro:enclosing_ellipse}.

Numerical experiments on various types of problems using a set of different parameters confirmed our analyses and showed that the bounds based on $\F(W)$ and $\F_{\MK}(K)$ lie very close together and almost always present a substantial improvement over the bounds obtained from $\F(K)$.
In most cases, the bounds based on $\F(W)$ and $\F_{\MK}(K)$ captured the convergence behavior closely, indicating that these spectral sets, in fact, appear to be the minimal regions of the complex plane governing the convergence behavior of the Krylov method, confirming the value of the analysis conducted in~\Cref{sect:basis-framework,sect:basis-metric,sect:fov}, which reflects the intrinsic block geometry inherited from $W$.
The new bounds seem to be suboptimal only on strongly nonnormal test examples, for which improved analyses and tighter bounds remain directions of further study.

\section*{Acknowledgment}
Both authors are members of the scientific network ``The $f(A)b$ulous scientific network on matrix functions and exponential integrators'', supported by Deutsche Forschungsgemeinschaft (DFG, German Research Foundation) under project number 566049107.

\appendix
\crefalias{section}{appendix}
\section{Numerically estimating the polynomial approximation error on the field of values}\label{appendix:poly_error}
In order to estimate~$E_{m-1}(\e^\zeta, \Omega)$ for some compact set $\Omega$ when evaluating the respective error bounds, we use the following approach: For any polynomial $p$, the function $p(\zeta)-\e^\zeta$ is holomorphic, so it attains its maximum modulus on the boundary of $\Omega$ (since in our setting $\Omega$ is the closure of a bounded domain). Thus, we discretize the boundary, using points $\zeta_1,\ldots,\zeta_N \in \partial\Omega$. We choose 
\[
    \eta:=\frac12\left(\min_j \operatorname{Re} \zeta_j+\max_j \operatorname{Re} \zeta_j\right)
      +\frac{i}{2}\left(\min_j \operatorname{Im} \zeta_j+\max_j \operatorname{Im} \zeta_j\right),
    \quad
    \rho:=\max_{1\le j\le N}|\zeta_j-\eta|
\]
and work in the scaled variables $\xi_j:=\frac{\zeta_j-\eta}{\rho}, j=1,\ldots,N$. The approximating polynomial is represented in a discrete Arnoldi basis generated on these scaled boundary nodes,
\[
    p(\zeta)=\sum_{k=0}^{m-1} \gamma_k q_k\!\left(\frac{\zeta-\eta}{\rho}\right),
    \quad
    Q_{j,k+1}=q_k(\xi_j),
\]
so that $Q\in\mathbb C^{N\times m}$ and $\gamma\in\mathbb C^m$. We collect the polynomial approximation errors $p(\zeta_i)-\e^{\zeta_i}$ in the vector $r(\gamma):=Q\gamma-x$, where $x=[\e^{\zeta_1},\dots,\e^{\zeta_N}]^T$. The exact minimax problem on the finite grid can then be written as
\[
    E_N := \min_{\gamma\in\mathbb C^m} \max_{1\le j\le N} \left| \sum_{k=0}^{m-1} \gamma_k q_k(\xi_j)-\e^{\zeta_j} \right| = \min_{\gamma\in\mathbb C^m} \max_{1\le j\le N} |r_j(\gamma)|,
\]
which is a discrete analogue of the continuous minimax problem on $\partial\Omega$. Equivalently, in epigraph form,
\[
\begin{aligned}
    E_N = \min_{\gamma\in\mathbb C^m,\,\Theta\in\mathbb R}\quad & \Theta,\\
    \text{subject to}\quad & |r_j(\gamma)|\le \Theta, \qquad j=1,\ldots,N .
\end{aligned}
\]
For fixed $\gamma$, the smallest admissible value of $\Theta$ is precisely $\max_j |r_j(\gamma)|$. 

Each constraint $|r_j(\gamma)|\le\Theta$ describes a disk in the complex plane, or equivalently a two-dimensional second-order cone constraint. To obtain a linear program, we approximate this disk by the $L$ half-plane constraints defining a regular circumscribed polygon. Thus,
with
\[
    \theta_\ell=\frac{2\pi\ell}{L}, \qquad \ell=0,\ldots,L-1,
\]
we impose
\[
    \operatorname{Re}\!\left(\e^{-i\theta_\ell}r_j(\gamma)\right) \le \Theta, \qquad j=1,\ldots,N,\quad \ell=0,\ldots,L-1 .
\]
This gives the real linear program
\[
\begin{aligned}
    \widehat E_{N,L}
    :=
    \min_{\gamma\in\mathbb C^m,\,\Theta\in\mathbb R}\quad & \Theta,\\
    \text{subject to}\quad
    &
    \operatorname{Re}\!\left(
        \e^{-i\theta_\ell}
        \left[
            \sum_{k=0}^{m-1}\gamma_k q_k(\xi_j)-\e^{\zeta_j}
        \right]
    \right)
    \le \Theta,\\
    &
    \qquad\qquad\qquad j=1,\ldots,N,\quad \ell=0,\ldots,L-1 .
\end{aligned}
\]
After splitting $\gamma$ into real and imaginary parts, all constraints are linear over the real variables. Since the circumscribed polygon has inradius $\Theta$ and circumradius $\frac{\Theta}{\cos(\pi/L)}$, the solution $\widehat\gamma$ of the linear program satisfies
\[
    \widehat E_{N,L} \leq E_N \leq \max_{1\le j\le N}|r_j(\widehat\gamma)| \leq \frac{\widehat E_{N,L}}{\cos(\pi/L)}.
\]
Thus the polyhedral approximation changes the discrete minimax value by at most the factor $\frac{1}{\cos(\pi/L)}$, which is close to one for the values of $L$ used in the computations.

To improve accuracy, the linear program is solved in a residual-correction form. Starting from a least-squares polynomial, we repeatedly solve a linear program for a correction to the current residual, rescaling the residual at each step to keep the linear program well conditioned. After solving on the current grid, the resulting polynomial is evaluated on a much finer boundary grid; points where the largest errors occur are added to the optimization grid, and the process is repeated. The final reported value is the maximum error of the computed polynomial on the fine validation grid.

The linear programs are solved in MATLAB with the built-in function \texttt{linprog} from the Optimization Toolbox. By default, we call \texttt{linprog} with the \texttt{dual-simplex} algorithm; in case of failure for a particular instance, the code retries with the \texttt{interior-point} algorithm. 

\bibliographystyle{myplain2-doi}
\bibliography{notebib}

\begin{thebibliography}{10}

\bibitem{almo25}
Awad~H. Al-Mohy.
\newblock \href{https://arxiv.org/abs/2509.26475}{Computing linear combinations
  of $\varphi$-function actions for exponential integrators}.
\newblock {ArXiv}:2509.26475 [math.{NA}], September 2025.

\bibitem{almo25a}
Awad~H. Al-Mohy.
\newblock \href{https://doi.org/10.3390/math13243985}{Shared-pole
  carath{\'e}odory--fej{\'e}r approximations for linear combinations of
  $\varphi$-functions}.
\newblock {\em Mathematics}, 13\penalty0 (24):\penalty0 3985, 2025.

\bibitem{alhi09}
Awad~H. Al-Mohy and Nicholas~J. Higham.
\newblock \href{https://doi.org/10.1137/09074721X}{A new scaling and squaring
  algorithm for the matrix exponential}.
\newblock {\em SIAM J. Matrix Anal. Appl.}, 31\penalty0 (3):\penalty0 970--989,
  2009.

\bibitem{alhi11}
Awad~H. Al-Mohy and Nicholas~J. Higham.
\newblock \href{https://doi.org/10.1137/100788860}{Computing the action of the
  matrix exponential, with an application to exponential integrators}.
\newblock {\em SIAM J. Sci. Comput.}, 33\penalty0 (2):\penalty0 488--511, 2011.

\bibitem{alli26}
Awad~H. Al-Mohy and Xiaobo Liu.
\newblock \href{https://doi.org/10.1137/25M1765262}{A scaling and recovering
  algorithm for the matrix $\varphi$-functions}.
\newblock {\em SIAM J. Sci. Comput.}, 48\penalty0 (2):\penalty0 A726--A747,
  2026.

\bibitem{arno51}
W.~Edwin Arnoldi.
\newblock The principle of minimized iterations in the solution of the matrix
  eigenvalue problem.
\newblock {\em Q. Appl. Math.}, 9\penalty0 (1):\penalty0 17--29, 1951.

\bibitem{bere09}
Bernhard Beckermann and Lothar Reichel.
\newblock \href{https://doi.org/10.1137/080741744}{Error estimates and
  evaluation of matrix functions via the {F}aber transform}.
\newblock {\em SIAM J. Numer. Anal.}, 47\penalty0 (5):\penalty0 3849–3883,
  2009.

\bibitem{benz20}
Michele Benzi.
\newblock \href{https://doi.org/10.1007/s40574-020-00249-2}{Some uses of the
  field of values in numerical analysis}.
\newblock {\em Boll. Unione Mat. Ital.}, 14:\penalty0 159--177, 2021.

\bibitem{bkt21}
Mike~A. Botchev, Leonid~A. Knizhnerman, and Eugene~E. Tyrtyshnikov.
\newblock \href{https://doi.org/10.1137/20m1375383}{Residual and restarting in
  {Krylov} subspace evaluation of the $\varphi$ function}.
\newblock {\em SIAM J. Sci. Comput.}, 43\penalty0 (6):\penalty0 A3733–A3759,
  2021.

\bibitem{ccz23}
Marco Caliari, Fabio Cassini, and Franco Zivcovich.
\newblock \href{https://doi.org/10.1016/j.cam.2022.114973}{{BAMPHI}:
  Matrix-free and transpose-free action of linear combinations of
  $\varphi$-functions from exponential integrators}.
\newblock {\em J. Comput. Appl. Math.}, 423:\penalty0 114973, 2023.

\bibitem{coma02}
S.M. Cox and P.C. Matthews.
\newblock \href{https://doi.org/10.1006/jcph.2002.6995}{Exponential time
  differencing for stiff systems}.
\newblock {\em J. Comput. Phys.}, 176\penalty0 (2):\penalty0 430--455, 2002.

\bibitem{crou07}
Michel Crouzeix.
\newblock \href{https://doi.org/10.1016/j.jfa.2006.10.013}{Numerical range and
  functional calculus in {Hilbert} space}.
\newblock {\em J. Funct. Anal.}, 244\penalty0 (2):\penalty0 668--690, 2007.

\bibitem{crpa17}
Michel Crouzeix and C{\'e}sar Palencia.
\newblock \href{https://doi.org/10.1137/17M1116672}{The numerical range is a
  $(1+\sqrt{2})$-spectral set}.
\newblock {\em SIAM J. Matrix Anal. Appl.}, 38\penalty0 (2):\penalty0 649--655,
  2017.

\bibitem{det23}
Pranab~J. Deka, Lukas Einkemmer, and Mayya Tokman.
\newblock \href{https://doi.org/10.1016/j.softx.2022.101302}{{LeXInt}:
  {Package} for exponential integrators employing {Leja} interpolation}.
\newblock {\em SoftwareX}, 21:\penalty0 101302, 2023.

\bibitem{dorr71}
Fred~W. Dorr.
\newblock \href{https://doi.org/10.1090/S0025-5718-1971-0297142-0}{An example
  of ill-conditioning in the numerical solution of singular perturbation
  problems}.
\newblock {\em Math. Comput.}, 25\penalty0 (114):\penalty0 271--283, 1971.

\bibitem{grt18}
St{\'e}phane Gaudreault, Greg Rainwater, and Mayya Tokman.
\newblock \href{https://doi.org/10.1016/j.jcp.2018.06.026}{{KIOPS}: {A} fast
  adaptive {Krylov} subspace solver for exponential integrators}.
\newblock {\em J. Comput. Phys.}, 372\penalty0 (1):\penalty0 236--255, 2018.

\bibitem{give52}
Wallace Givens.
\newblock \href{https://doi.org/10.1090/S0002-9939-1952-0047004-9}{Fields of
  values of a matrix}.
\newblock {\em Proc. Amer. Math. Soc.}, 3:\penalty0 206--209, 1952.

\bibitem{gowi76}
Gene~H. Golub and James~H. Wilkinson.
\newblock Ill-conditioned eigensystems and the computation of the {Jordan}
  canonical form.
\newblock {\em SIAM Rev.}, 18\penalty0 (4):\penalty0 578--619, 1976.

\bibitem{grcar89}
Joseph~F. Grcar.
\newblock Operator coefficient methods for linear equations.
\newblock Technical Report SAND89-8691, Sandia National Laboratories,
  Albuquerque, NM, 1989.

\bibitem{gura97}
Karl~E. Gustafson and Duggirala K.~M. Rao.
\newblock \href{http://doi.org/10.1007/978-1-4613-8498-4}{{\em Numerical Range:
  The Field of Values of Linear Operators and Matrices}}.
\newblock Springer, New York, 1997.
\newblock xiv+190 pp.
\newblock ISBN 978-1-4613-8498-4.

\bibitem{gusc21}
Stefan Güttel and Marcel Schweitzer.
\newblock \href{https://doi.org/10.1137/20m1351072}{A comparison of
  limited-memory {Krylov} methods for {Stieltjes} functions of {Hermitian}
  matrices}.
\newblock {\em SIAM J. Matrix Anal. Appl.}, 42\penalty0 (1):\penalty0 83–107,
  2021.

\bibitem{hade92}
Uffe Haagerup and Pierre de~la Harpe.
\newblock The numerical radius of a nilpotent operator on a {Hilbert} space.
\newblock {\em Proc. Amer. Math. Soc.}, 115\penalty0 (2):\penalty0 371--379,
  1992.

\bibitem{high:FM}
Nicholas~J. Higham.
\newblock \href{http://doi.org/10.1137/1.9780898717778}{{\em Functions of
  Matrices: {Theory} and Computation}}.
\newblock SIAM, Philadelphia, PA, USA, 2008.
\newblock xx+425 pp.
\newblock ISBN 978-0-898716-46-7.

\bibitem{holu97}
Marlis Hochbruck and Christian Lubich.
\newblock \href{https://doi.org/10.1137/S0036142995280572}{On {Krylov} subspace
  approximations to the matrix exponential operator}.
\newblock {\em SIAM J. Numer. Anal.}, 34\penalty0 (5):\penalty0 1911--1925,
  1997.

\bibitem{hls98}
Marlis Hochbruck, Christian Lubich, and Hubert Selhofer.
\newblock \href{https://doi.org/10.1137/S1064827595295337}{Exponential
  integrators for large systems of differential equations}.
\newblock {\em SIAM J. Sci. Comput.}, 19\penalty0 (5):\penalty0 1552--1574,
  1998.

\bibitem{hoos10}
Marlis Hochbruck and Alexander Ostermann.
\newblock \href{https://doi.org/10.1017/S0962492910000048}{Exponential
  integrators}.
\newblock {\em Acta Numerica}, 19:\penalty0 209–286, 2010.

\bibitem{hojo13}
Roger~A. Horn and Charles~R. Johnson.
\newblock {\em Matrix Analysis}.
\newblock 2nd edition, Cambridge University Press, Cambridge, UK, 2013.
\newblock xviii+643 pp.
\newblock ISBN 978-0-521-83940-2.

\bibitem{kms53}
Mark Kac, W.~L. Murdock, and Gabor Szeg{\"o}.
\newblock \href{https://doi.org/10.1512/iumj.1953.2.52034}{On the eigen-values
  of certain {Hermitian} forms}.
\newblock {\em J. Rational Mech. Anal.}, 2\penalty0 (4):\penalty0 767--800,
  1953.

\bibitem{kaha66}
William Kahan.
\newblock Numerical linear algebra.
\newblock {\em Canad. Math. Bull.}, 9:\penalty0 757--801, 1966.

\bibitem{krsc25}
Emil Krieger and Marcel Schweitzer.
\newblock A general framework for {K}rylov {ODE} residuals with applications to
  randomized {K}rylov methods.
\newblock {arXiv}:2510.17538, 2025.

\bibitem{lanc50}
C.~Lanczos.
\newblock An iteration method for the solution of the eigenvalue problem of
  linear differential and integral operators.
\newblock {\em J.\ Res.\ Natl.\ Stand.}, 45:\penalty0 255--282, 1950.

\bibitem{maha02}
John~C. Mason and David~C. Handscomb.
\newblock \href{http://doi.org/10.1201/9781420036114}{{\em Chebyshev
  Polynomials}}.
\newblock Chapman and Hall/CRC, Boca Raton, FL, 2002.

\bibitem{miwr05}
Borislav~V. Minchev and Will~M. Wright.
\newblock \href{https://cds.cern.ch/record/848122/files/cer-002531456.pdf}{A
  review of exponential integrators for first order semi-linear problems}.
\newblock Tech. Report 2/05, Norwegian University of Science and Technology,
  Trondheim, Norway, 2005.

\bibitem{niwr12}
Jitse Niesen and Will~M. Wright.
\newblock \href{https://doi.org/10.1145/2168773.2168781}{Algorithm 919: A
  {Krylov} subspace algorithm for evaluating the $\varphi$-functions appearing
  in exponential integrators}.
\newblock {\em ACM Trans. Math. Softw.}, 38\penalty0 (3):\penalty0 22:1--22:19,
  2012.

\bibitem{penr55}
Roger Penrose.
\newblock \href{https://doi.org/10.1017/S0305004100030401}{A generalized
  inverse for matrices}.
\newblock {\em Proc. Cambridge Philos. Soc.}, 51\penalty0 (3):\penalty0
  406--413, 1955.

\bibitem{ruth52}
Daniel~Edwin Rutherford.
\newblock \href{https://doi.org/10.1017/S0080454100007111}{Some continuant
  determinants arising in physics and chemistry---{II}}.
\newblock {\em Proc. Royal Soc. Edin.}, 63\penalty0 (A):\penalty0 232--241,
  1952.

\bibitem{saad92}
Yousef Saad.
\newblock \href{https://doi.org/10.1137/0729014}{Analysis of some {Krylov}
  subspace approximations to the matrix exponential operator}.
\newblock {\em SIAM J. Numer. Anal.}, 29\penalty0 (1):\penalty0 209--228, 1992.

\bibitem{sidj98}
Roger~B. Sidje.
\newblock \href{https://doi.org/10.1145/285861.285868}{Expokit: A software
  package for computing matrix exponentials}.
\newblock {\em ACM Trans. Math. Software}, 24\penalty0 (1):\penalty0 130--156,
  1998.

\bibitem{tref92}
Lloyd~N. Trefethen.
\newblock Pseudospectra of matrices.
\newblock In {\em Numerical Analysis 1991}, D.~F. Griffiths and G.~A. Watson,
  editors, volume 260 of {\em Pitman Research Notes in Mathematics Series},
  Longman Scientific and Technical, Harlow, UK, 1992, pages 234--266.

\bibitem{trem05}
Lloyd~N. Trefethen and Mark Embree.
\newblock {\em Spectra and Pseudospectra: The Behavior of Nonnormal Matrices
  and Operators}.
\newblock Princeton University Press, Princeton, NJ, 2005.
\newblock ISBN 978-0-691-11946-5.

\bibitem{varm14}
Vipin~Kerala Varma.
\newblock \href{https://doi.org/10.48550/arXiv.1410.4932}{Conformal map and
  harmonic measure of the {B}unimovich stadium}.
\newblock Technical report, 2014.
\newblock arXiv:1410.4932.

\bibitem{waye17}
Hao Wang and Qiang Ye.
\newblock \href{https://doi.org/10.1137/16M1063733}{Error bounds for the krylov
  subspace methods for computations of matrix exponentials}.
\newblock {\em SIAM J. Matrix Anal. Appl.}, 38\penalty0 (1):\penalty0 155--187,
  2017.

\end{thebibliography}

\end{document}